\documentclass[11pt, final]{article}
\usepackage{main}

\title{\bf{Lipschitz approximation for general almost area minimizing currents}}
\author{\Large{Reinaldo Resende}}
\affil{\normalsize \textit {Carnegie Mellon University}}

\begin{document}

\date{}

\maketitle

\begin{abstract}
We prove a Lipschitz approximation with superlinear error terms for integral currents $\omega$-minimizing the area functional, where $\omega$ is a modulus of continuity {satisfying a Dini condition}. We also present an almost monotonicity result for the density of these general almost area minimzing integral currents.
\end{abstract}


\section{Introduction}

The regularity theory is a widely spread theme in mathematics. It splits into various branches, for example, the regularity theory of PDEs, the regularity theory of minimal surfaces, etc. One of the first and most famous problems is the existence and regularity of minimal surfaces, in other words, the existence and regularity of objects that minimize the area functional over some class of admissible competing surfaces. The classes of admissible objects where the area functional may be defined can be genuinely different, giving rise to quite different approaches to geometric variational problems. For example, the classical one considers the functional area defined in the class of smooth submanifolds of a fixed ambient Riemannian manifold. 

We quickly realize that in this smooth context, the minimization problem for the area functional manifests a lack of compactness that naturally leads to considering and introducing objects that play the role of generalized smooth surfaces in the same guise of what is done when weak solutions are introduced to work around the weaknesses of the space of classical solutions of PDEs. Over the years, many generalizations have been proposed, for instance, the theory of Caccioppoli sets, varifolds, currents, flat chains, etc. 


In this note, we will focus on the theory of integral currents that minimize (in a relaxed sense) the area functional. According to the definition given by \textsc{I. Tamanini}, \cite[Eq. 1.2]{tamanini1984regularity}, 
we work with the relaxed minimality condition, which we call the \(\omega\)-almost minimality condition (see Definition \ref{D}). This condition is natural since it arises from practical problems as \textsc{F. Maggi} noticed in his book \cite[Chapter III]{maggi2012sets}, where for the codimension \(1\) setting, he considers an almost minimality condition that is a particular case of the {\(\Omega\)}-almost minimality. 


A natural question to ask is: do these generalized surfaces have good regularity properties provided they minimize area?. Aiming at answering this question, in arbitrary dimension and codimension, in the setting of integral currents, \textsc{F. Almgren Jr.} has introduced his long and intricate, but still rich and beautiful, program in \cite{Alm} to prove regularity results for interior points of area minimizing integral currents. He stated that the singular set has Hausdorff dimension at most \(m-2\). His theory was revisited by \textsc{C. De Lellis and E. Spadaro},  in a series of works (see \cite{DS3}) where they furnished a different approach using new techniques of geometric analysis which give a much shorter proof and they also strengthened the main result. More recently, in \cite{skorobogatova2021upper}, \textsc{A. Skorobogatova} improved \textsc{Almgren's} estimate, she proves that the upper Minkowski dimension of the interior singular set is at most \(m-2\). 


In this article, we aim to perform the first part of the regularity program used in the aforementioned works, which are (almost-) monotonicity results, the strong Lipschitz approximation, and the strong excess decay at interior points. We aim at proving these first part of the framework for \(\omega\)-almost area minimizing integral currents {where $\omega$ is a modulus of continuity satisfying a Dini condition}. 
{Moreover, we give an example (Example \ref{example}) of a current that satisfies the \(\omega\)-almost minimality condition and it is not covered by the \(\AM\)-almost minimality condition in \cite{DSS2} and \cite{maggi2012sets}.}

\section{Preliminaries}

The goal of this section is to set standard notations on currents theory that will be used throughout this paper.


We use \(\ball{p}{r}\subset\R^{m+n}\) for the open balls centered at \(p\in\R^{m+n}\) and of radius \(r\in]0,+\infty[\) of the ambient space \(\R^{m+n}\), and we fix \(\pi_0 := \R^m\times\{0\}\subset\R^{m+n}\). For any linear subspace \(\pi \subset \R^{m+n}, \pi^{\perp}\) is its orthogonal complement in \(\R^{m+n}\), \(\bp_\pi\) is the orthogonal projection onto \(\pi,\) and \(\bp := \bp_{\pi_0}\). 

We define the \textbf{tilted disk} \(\tball{p}{r}{\pi} := \ball{p}{r}\cap(p+\pi)\) and \(\cyltilted{p}{r}{\pi}\) the \textbf{tilted cylinder} as the set \(\left\{(x+y): x \in \tball{p}{r}{\pi}, y \in \pi^{\perp}\right\}\). We also set \(\cyl{p}{r}:= \cyltilted{p}{r}{\pi_0}\) and \(\baseball{p}{r}:= \tball{p}{r}{\pi_0}\). Moreover, \(\ocyl{r}:=\cyl{0}{r}=\cyltilted{0}{r}{\pi_0}\).

We also assume that each linear subspace \(\pi\) of \(\R^{m+n}\) is oriented by a \(k\)-vector \(\vec{\pi}:=v_1 \wedge \cdots \wedge v_k\), where \((v_i)_{i\in\{1,...,k\}}\) is an orthonormal base of \(\pi\) and, with an abuse of notation, we write \(\left|\pi_2-\pi_1\right|\) standing for \(\left|\vec{\pi}_2-\vec{\pi}_1\right|\), where \(|\cdot|\) is the norm associated to the canonical inner product of \(k\)-vectors.

For any $s\in[0,+\infty[$ we also set \(\cH^s\) as the \(s\)-dimensional Hausdorff measure in \(\R^{m+n}\). We recall the definition of density of a given $T\in\mathcal{D}_m(U)$, where $U\subseteq\R^{m+n}$ is an open set and $\mathcal{D}_m(U)$ is the set of $m$-dimensional current in $U$ at a given point $p\in\R^{m+n}$. We say that $\Theta^m(T,p)\in[0,+\infty]$ is the \textbf{{$m$-dimensional density of T at $p$}}, if
\begin{equation*} \Theta^m(T,p) = \lim_{r\to 0}\frac{\|T\|(\ball{p}{r})}{{\omega_m r^m}},\end{equation*}
whenever the limit exists, {where $\omega_m$ is the measure of the $m$-dimensional unit ball}.  

For standard notations and classical results on the theory of currents which will be used in this note, we refer the reader to the classical treatise of geometric measure theory \cite{Fed}. {In fact, given a current $T$, we denote by $\partial T$ the boundary of $T$, $\spt(T)$ its support, $\|T\|$ its associated total variation measure, and we use $\mathbf{I}_k(\R^{h})$ for the space of integral $k$-currents in $\R^{h}$ for $h\geq k$ both natural numbers}. For the theory of multi-valued maps, we refer the reader to \cite{DS2}. 

\begin{defi}[Excess and height]
Given an { \(m\)-dimensional integral }current \(T\) in \(\mathbb{R}^{m+n}\), i.e., \(T\in\mathbf{I}_m(\R^{m+n})\) and \(m\)-planes \(\pi,\) and \(\pi^{\prime}\), we define \textbf{the excesses of \(T\) in balls and cylinders} as
\begin{equation*}\bE\left(T, \ball{p}{r}, \pi\right):=\frac{1}{2\omega_m r^m} \int_{\ball{p}{r}}|\vec{T}-\vec{\pi}|^2 d\|T\|,\end{equation*}
\begin{equation*}\bE\left(T, \cyltilted{p}{r}{\pi}, \pi^{\prime}\right):=\frac{1}{2\omega_m r^m} \int_{\cyltilted{p}{r}{\pi}}\left|\vec{T}-\vec{\pi}^{\prime}\right|^2 d\|T\|,\end{equation*}
we will use the shorthand notation \(\bE\left(T, \cyltilted{p}{r}{\pi}\right)\) for \(\bE\left(T, \cyltilted{p}{r}{\pi}, \pi\right)\). We define also \textbf{the height function in a set \(A \subset \mathbb{R}^{m+m}\) with respect to the $m$-plane $\pi$} as
\begin{equation*}
\mathbf{h}(T, A, \pi):=\sup _{x, y \in \operatorname{spt}(T) \cap A}\left|\mathbf{p}_{\pi^{\perp}}(x)-\mathbf{p}_{\pi^{\perp}}(y)\right| .
\end{equation*}
\end{defi}

\begin{defi}[Optimal planes for the excess]
We say that an \(m\)-dimensional plane \(\pi\) \textbf{optimizes the excess of \(T\) in a ball \(\ball{p}{r}\)} if
\begin{equation}\label{D:excess}
    \bE\left(T, \ball{p}{r}\right):={\inf} _{\pi^{\prime}} \bE\left(T, \ball{p}{r},\pi^{\prime}\right)=\bE\left(T, \ball{p}{r}, \pi\right).
\end{equation} 
\end{defi}

Observe that in general the plane optimizing the excess is not unique and \(\mathbf{h}\left(T, \ball{p}{r}, \pi\right)\) might depend on the optimizer \(\pi\). 

\begin{defi}[Optimal planes]
A $m$-plane \(\pi\) is called an \textbf{optimal plane}, if it optimizes the height function among the planes that are optimal for the excess, i.e.,
\begin{equation*}\bh\left(T, \ball{p}{r}, \pi\right)={\inf} \left\{\bh\left(T, \ball{p}{r}, \pi^{\prime}\right): \pi^{\prime}\right. \text{ satisfies }\ \eqref{D:excess}\left.\right\}=: \mathbf{h}\left(T, \ball{p}{r}\right).\end{equation*}

Henceforth, \(\bh\left(T, \cyltilted{p}{r}{\pi}\right)\) will stand for \(\bh\left(T, \cyltilted{p}{r}{\pi}, \pi\right)\).
\end{defi} 

Lastly, we {give} the definition of \(\omega\)-almost minimality {as it is defined in \cite{tamanini1984regularity}. After that, we} briefly discuss the high level of generality that this condition represents.

\begin{defi}[Almost minimality condition]\label{D}
Let \(T\) be an \(m\)-dimensional integral current in \(\R^{m+n}\), i.e., \(T\in\mathbf{I}_m(\R^{m+n})\). We say that \(T\) is \textbf{\(\omega\)-almost area minimizing}, if, for each \(p\in\spt(T)\setminus\spt(\partial T)\), there exist \(s_{\omega}>0\) and an absolutely continuous nondecreasing function \(\omega: (0, {s_{\omega}}) \to (0,+\infty)\) such that \(\omega(s)= o(1)\) when $s\to0^+$, and 
\begin{equation*} \| T\|(\ball{p}{s}) \leq (1 + \omega(r))\| T+\partial S\|(\ball{p}{s}), \quad \forall s\in (0,s_{\omega}), \quad \forall S\in\mathbf{I}_{m+1}(\ball{p}{s}).
\end{equation*}

In the special case that \(\omega(s) = \bA s^{\balpha}\) for some \(\bA \geq 0\) and \(\balpha\in(0,1]\), we say that \(T\) is \textbf{\(\AM\)-almost area minimizing} where $\mathbf{s} = s_\omega$. If \(\omega \equiv 0\), we say that \(T\) is \textbf{area minimizing}.
\end{defi}

We now provide an example of an \(\omega\)-almost minimizer which does not belong to the class of \(\AM\)-almost minimizer.

\begin{ex}
We consider the function
\begin{equation*}
\begin{array}{crcl}
f : & (0,1)\subset\R & \to & (0,1)\subset\R \\
 & t & \mapsto & \int_{0}^t \left(\ln(\frac{e}{s})\right)^{-1}\diff s.
\end{array}
\end{equation*}

We have that \(f(0)=0\) and we set \(f(t) = f(-t), \forall t\in (-1,0)\). So, if we consider the Caccioppoli set given by \(E:=\operatorname{epi}(f)\) and \(Q_t:=(-t,t)^2\subset \R^2\), we have that 
\begin{equation*}
\begin{aligned}
\mathcal{P}(E,Q_t) - t = 2\int_0^t\left(\sqrt{1+(\ln(\frac{e}{s}))^{-2}}-1\right)ds \approx t (\ln(\frac{e}{t}))^{-2},
\end{aligned}
\end{equation*}

So, we have that the \(1\)-current induced by \(E\) is \(\omega\)-almost minimizer with \(\omega(t) = (\ln(\frac{e}{t}))^{-2}.\)
\end{ex}

\section{The \texorpdfstring{\(\omega\)}{Lg}-almost monotonicity formulas}

{Before stating the almost monotonicity of the mass ratio for \(\omega\)-almost area minimizers (Proposition \ref{P:almost-monotonicity-times}), we state the almost monotonicity result when we consider the definition of almost minimality given in Definition \ref{D2}. Furthermore, having the following almost moniticity of the almost ratio allows us to carry out the program in this article with the \(\Omega\)-almost area minimizers.
\begin{defi}[\(\Omega\)-almost minimality condition]\label{D2}
Let \(T\) be an \(m\)-dimensional integral current in \(\R^{m+n}\), i.e., \(T\in\mathbf{I}_m(\R^{m+n})\). We say that \(T\) is \textbf{\(\Omega\)-almost area minimizing}, if, for each \(p\in\spt(T)\setminus\spt(\partial T)\), there exist \(s_{\Omega}>0\) and an absolutely continuous nondecreasing function \(\Omega: (0, {s_{\Omega}}) \to (0,+\infty)\) such that \(\Omega(s)= o(1)\) when $s\to0^+$, and 
\begin{equation*} \| T\|(\ball{p}{s}) \leq \| T+\partial S\|(\ball{p}{s}) + \Omega(s)s^m, \quad \forall s\in (0,s_{\Omega}), \quad \forall S\in\mathbf{I}_{m+1}(\ball{p}{s}).\end{equation*}
\end{defi}
Note that, this almost monotonicity result comes} with an additive error term which is the analogous of \cite[Proposition 2.1]{DSS1} for {this $\Omega$-almost minimality condition}. Henceforth, we will denote by \((z-p)^{\perp}\) the projection of the vector \(z-p\) onto the orthogonal complement of the approximate tangent to \(T\) at \(z\).

Following \textsc{I. Tamanini}'s ideas, \cite{tamanini1984regularity}, we give an example to justify the different definitions of almost area minimality that we mentioned.

\begin{ex}\label{example}
We \emph{identify} \(T\) to the Caccioppoli set \(E\subset\R^{m+1}\) which is a minimizer of the following variational problem:
\begin{equation}\label{mean-curvature}
\inf\left \{\mathcal{P}(F, \ball{p}{s}) + \int_{F}H(q)\diff q : F \ \text{is a Caccioppoli set and} \ F\Delta E\subset\subset \ball{p}{s} \right \}, 
\end{equation}
where \(H\) is a prescribed mean curvature function that belongs to \(\mathrm{L}^{{\ell}}(\R^{m+1})\) with the crucial condition that \({{\ell}} > {m+1}\) and \(\mathcal{P}\) denotes the perimeter measure of Caccioppoli sets, see \cite{maggi2012sets}. Given any \(F\) as in \eqref{mean-curvature}, thanks to \({{\ell}}>{m+1}\), we must apply the H\"older inequality to derive
\begin{equation*}
\begin{aligned}
\mathcal{P}(E, \ball{p}{s}) - \mathcal{P}(F, \ball{p}{s}) &\leq \int_{E\Delta F}|H(x)|\diff x \\
&\leq \| H\|_{\mathrm{L}^{{\ell}}(\ball{p}{s})}{\cH^{m+1}}(\ball{p}{s})^{1-\frac{1}{{\ell}}} \\
&\leq \omega_{{m+1}}^{1-\frac{1}{{\ell}}}\| H\|_{\mathrm{L}^{{\ell}}(\ball{p}{s})} s^{{m+1}-\frac{{m+1}}{{{\ell}}}}.
\end{aligned}
\end{equation*}

{Therefore, by definition, one sees that \(T\) is an $\Omega$-almost area minimzers in $\mathbb{R}^{m+1}$ with}
\begin{equation*}
\Omega(s) = c_0(m)\omega_{{m+1}}^{1-\frac{1}{{{\ell}}}}\| H\|_{\mathrm{L}^{{\ell}}(\R^{m+1})}s^{1-\frac{m+1}{\ell}}\text{ and }s_\Omega=+\infty.
\end{equation*}
\end{ex}

\begin{prop}[\(\Omega\)-almost monotonicity formula]
Let \(T \in \mathbf{I}_{m}\left(\mathbb{R}^{m+n}\right)\) be an \(\Omega\)-almost minimizer and \(p\in \operatorname{spt}(T) \backslash \operatorname{spt}(\partial T)\). There are constants \(C_0, r_0 > 0\) such that
\begin{equation}\label{monotonicity-formula}
  \begin{aligned}
\int_{\ball{p}{r} \backslash \ball{p}{s}} \frac{\left|(z-p)^{\perp}\right|^{2}}{|z-p|^{m+2}} d\|T\|(z) \leq C_0\left(\frac{\|T\|\left(\ball{p}{r}\right)}{\omega_mr^m}-\frac{\|T\|\left(\ball{p}{s}\right)}{\omega_{m} s^{m}}+\int_s^r\frac{\Omega(\rho)}{\rho}\diff\rho\right),
\end{aligned} 
\end{equation}
for all \(0<s<r<r_0<{s_{\Omega}}\). Furthermore, whenever $\Omega$ satisfies a Dini condition, i.e., $\int_0^r \frac{\Omega(\rho)}{\rho}\diff\rho<+\infty$, the function $r \mapsto \frac{\|T\|\left(\ball{p}{r}\right)}{\omega_{m} r^{m}}+\int_0^r \frac{\Omega(\rho)}{\rho}\diff\rho$ is nondecreasing. 
\end{prop}
\begin{proof}
{Assume that $r_0$ small enough, so that $\partial T\res \ball{p}{r_0} = 0$ and take $r<r_0$. }We define the integral current
\begin{equation*}W := {p} \cone \partial\left(T\res\ball{p}{r}\right),\end{equation*}
{which is an $m$-dimensional integral current for a.e. $r>0${, see \cite[Chapter 6, Section 5]{simon2014introduction}}. We now test $W$ in} the \(\Omega\)-almost minimality of $T$ {(which is possible thanks to \cite[4.2.10]{Fed})} to obtain
\begin{equation}\label{min-for-W}
{\|T\|\left(\ball{p}{r}\right) \leq \|W\|\left(\ball{p}{r}\right) + \Omega(r)r^{m} {=} \frac{r}{m}\mathbf{M}\left(\partial (T\res\ball{p}{r})\right)+ \Omega(r)r^m,\text{ for a.e. $r>0$}},
\end{equation}
where $\mathbf{M}$ denotes the total mass of an integral current{, see \cite[Proposition 1.34]{tasnady2006uniqueness} for the classical equality.} 
We now set the mass function \(\fm(r):=\|T\|\left(\ball{p}{r}\right)\) and observe that \(\fm\) is a nondecreasing function and thus it is a function of bounded variation. We can decompose its distributional derivative \(D\fm\), which is a nonnegative measure, as \(D\fm=\fm^{\prime}\cH^1 + \mu_{s}\) and \(\mu_{s}\) is the singular part of \(D\fm\). In \eqref{min-for-W}, we multiply \(mr^{-m-1}\) to obtain
\begin{equation}\label{almost-min-proof}
-\frac{1}{r^{m}} \mathbf{M}\left(\partial (T\res\ball{p}{r})\right) \leq - \frac{{m}\fm(r)}{r^{m+1}} + m\frac{\Omega(r)}{r},\text{ for a.e. $r>0$},
\end{equation}
{Using the notation $p(r) = r^{m+1}$, note that, for a.e. $r>0$, one has that
\begin{equation}\label{derivate-frakm}
    \frac{-m\fm}{p}\cH^1 = D\biggl( \frac{\fm{(r)}}{r^m} \biggr) - \frac{D\fm(r)}{r^m}.
\end{equation}}
It is now easy to see, by integrating \eqref{almost-min-proof} on the interval \((s, t]\), take advantage of the right-continuity of $\rho \to \fm(\rho)/\rho^m$, where \(s_{\Omega} > r_0>t>s\), {and \eqref{derivate-frakm} that we reach
\begin{equation*}
\begin{aligned}
    -\int_{s}^{t} \frac{\mathbf{M}\left(\partial (T\res\ball{p}{\rho})\right)}{\rho^{m}} \diff\cH^1(\rho) &\leq \frac{\fm(t)}{t^{m}}-\frac{\fm(s)}{s^{m}} - \int_s^t \frac{d \mu_{s}(\rho)}{\rho^{m}} - \int_s^t \frac{\fm^{\prime}(\rho)}{\rho^{m}} \diff\cH^1(\rho) \\
    &\quad +m\int_s^t \rho^{-1}\Omega(\rho)\diff\rho.
\end{aligned}
\end{equation*}
This leads to
\begin{equation*}
\begin{aligned}
\underbrace{\int_s^t \frac{d \mu_{s}(\rho)}{\rho^{m}} }_{I^{s}} + \underbrace{\int_{s}^{t} \frac{\fm^{\prime}(\rho)-\mathbf{M}\left(\partial (T\res\ball{p}{\rho})\right)}{\rho^{m}} \diff\cH^1(\rho)}_{I^{a}} &\leq \frac{\fm(t)}{t^{m}}-\frac{\fm(s)}{s^{m}}+m\int_s^t \rho^{-1}\Omega(\rho)\diff\rho.
\end{aligned}
\end{equation*}}

In the proof of \cite[Proposition 2.1]{DSS1}, it is shown that \(I:=I^{s}+I^{a}\) bounds, up to a dimensional constant, the left-hand side of \eqref{monotonicity-formula} without the use of any minimality condition. So, it finishes the proof of our result.
\end{proof}

We now state a second almost monotonicity result for the \(\omega\)-almost minimality condition {(Definition \ref{D})} which now comes with multiplicative error terms, these result is the analogous for interior points of \cite[Proposition 2.3]{hirsch2019uniqueness} which is stated to boundary points and only for the special case \(\omega(r) = C r^{\alpha}\). 

\begin{prop}[\(\omega\)-almost monotonicity formula]\label{P:almost-monotonicity-times}
Let \(T\in\mathbf{I}_m(\R^{m+n})\) be an \(\omega\)-almost minimizer, \(p\in\spt(T)\setminus\spt(\partial T)\){, and assume that $D(\rho):= \int_0^{\rho}\frac{\omega(\rho)}{\rho}\mathrm{d}\rho$ is such that $D(s_\omega)<+\infty$,} then there exists a constant \(r_1>0\), such that
\begin{equation}\label{almost-monotonicity-times}
\int_{\ball{p}{t} \setminus \ball{p}{s}} e^{m{ D(|z-p|)}}\frac{\left|(z-p)^{\perp}\right|^{2}}{2|z-p|^{m+2}} d\|T\|(z) \leq e^{m{D(t)}} \frac{\|T\|\left(\ball{p}{t}\right)}{t^{m}}-e^{m{D(s)}} \frac{\|T\|\left(\ball{p}{s}\right)}{s^{m}},
\end{equation}
for every \(0<s<t<r_1<{s_{\omega}}\).
\end{prop}
\begin{proof} 
We start defining 
\begin{equation*}S: = T\res\ball{p}{r} \quad \text{ and } \quad W := {p} \cone \partial S.\end{equation*}

{Note that, by slicing theory, $W$ is an integral $m$-current for a.e. $r<r_1$ where $r_1$ is fixed such that $\partial T\res\ball{p}{r_1}=0$. }By the \(\omega\)-almost minimizing property of \(T\) {(which we can apply thanks to \cite[4.2.10]{Fed})}, we deduce for {a.e.} \(r<{r_1<s_{\omega}}\) that
\begin{equation}\label{E:almost-mon-mass-bound}
\begin{aligned}
\|T\|\left(\ball{p}{r}\right) \leq (1 + \omega(r)) \|W\|(\ball{p}{r}) {=} (1+\omega(r))\frac{r}{m} \|\partial S\|(\ball{p}{r}),
\end{aligned}
\end{equation}
see \cite[Chapter 6, Section 5]{simon2014introduction} { and \cite[Proposition 1.34]{tasnady2006uniqueness} for the equality.} For a.e. \(\rho\leq r_{1} <{s_{\omega}}\), 
we conclude that 
\begin{equation}\label{conta-longa}
\begin{aligned}
\frac{\diff}{\diff\rho}\left(\frac{\|T\|\left(\ball{p}{\rho}\right)}{\rho^m}\right) & =  \frac{-m\|T\|\left(\ball{p}{\rho}\right)}{\rho^{m+1}}+\frac{\|T\|^{\prime}\left({\ball{p}{\rho}}\right)}{\rho^m}\\
& = \frac{-m\|T\|\left(\ball{p}{\rho}\right)}{\rho^{m+1}}+\frac{\|T\|^{\prime}\left({\ball{p}{\rho}}\right)}{\rho^m} + \frac{m\|T\|\left(\ball{p}{\rho}\right)}{\left(1+\omega(\rho)\right)\rho^{m+1}} - \frac{m\|T\|\left(\ball{p}{\rho}\right)}{\left(1+\omega(\rho)\right)\rho^{m+1}}\\
&= \frac{m\|T\|(\ball{p}{\rho})}{\rho^{m+1}}\left(\frac{1}{1+\omega(\rho)} - 1 \right) + \underbrace{\frac{\|T\|^{\prime}\left({\ball{p}{\rho}}\right)}{\rho^m} - \frac{m\|T\|\left(\ball{p}{\rho}\right)}{\left(1+\omega(\rho)\right)\rho^{m+1}}}_{I_1},\\
\end{aligned}
\end{equation}
{where $\|T\|^{\prime}(\ball{p}{\rho})$ represents the distributional derivative of $f(t):=\|T\|(\ball{p}{\rho})$}. We only use \eqref{E:almost-mon-mass-bound} to bound \(I_1\) from below as follows
\begin{equation}\label{I_1}
I_1 \geq \frac{1}{\rho^m} \left( \|T\|^{\prime}(\ball{p}{\rho}) - \|\partial S\|(\ball{p}{\rho}) \right).
\end{equation}

Therefore, by \eqref{conta-longa} and \eqref{I_1}, we obtain that 
\begin{equation*}
\begin{aligned}
\frac{\diff}{\diff\rho}\left(\frac{\|T\|\left(\ball{p}{\rho}\right)}{\rho^m}\right) &\geq \frac{m\|T\|(\ball{p}{\rho}}{\rho^{m+1}}\left(\frac{1}{1+\omega(\rho)} - 1 \right) + \frac{1}{\rho^m} \left( \|T\|^{\prime}(\ball{p}{\rho}) - \|\partial S\|(\ball{p}{\rho}) \right) .
\end{aligned}   
\end{equation*} 

Since \(\omega\) is absolutely continuous, by the Lebesgue differentiation theorem, we have that \(\omega\) is differentiable almost everywhere on \((0,r_1)\). {If we denote by $D(\rho) = \int_0^\rho\frac{\omega(t)}{t}\mathrm{d}t$}, hence the latter {inequality} provides, for a.e. \(\rho\in (0,r_1)\),
{
\begin{equation}\label{differential-equation2}
\begin{aligned}
\frac{\diff}{\diff\rho}\left(e^{m D(\rho)}\frac{\|T\|\left(\ball{p}{\rho}\right)}{\rho^m}\right) &= m D^{\prime}(\rho)e^{m D(\rho)}\frac{\|T\|\left(\ball{p}{\rho}\right)}{\rho^m} + e^{m D(\rho)}\frac{\diff}{\diff\rho}\left(\frac{\|T\|\left(\ball{p}{\rho}\right)}{\rho^m}\right) \\
&\geq me^{m D(\rho)}\left(D^\prime(\rho) + \frac{1}{\rho} \left(\frac{1}{1+\omega(\rho)} - 1 \right)\right)\frac{\|T\|\left(\ball{p}{\rho}\right)}{\rho^m} \\
&\quad + e^{mD(\rho)}\frac{1}{\rho^m} \left( \|T\|^{\prime}(\ball{p}{\rho} - \|\partial S\|(\ball{p}{\rho}) \right) \\
&\geq me^{m D(\rho)}\left(D^\prime(\rho) - \frac{\omega(\rho)}{\rho(1+\omega(\rho))}\right)\frac{\|T\|\left(\ball{p}{\rho}\right)}{\rho^m} \\
&\quad + e^{mD(\rho)}\frac{1}{\rho^m} \left( \|T\|^{\prime}(\ball{p}{\rho} - \|\partial S\|(\ball{p}{\rho}) \right) \\
&\geq me^{m D(\rho)}\left(D^\prime(\rho) - \frac{\omega(\rho)}{\rho}\right)\frac{\|T\|\left(\ball{p}{\rho}\right)}{\rho^m} \\
&\quad + e^{mD(\rho)}\frac{1}{\rho^m} \left( \|T\|^{\prime}(\ball{p}{\rho} - \|\partial S\|(\ball{p}{\rho}) \right) \\
&\geq e^{mD(\rho)}\frac{1}{\rho^m} \left( \|T\|^{\prime}(\ball{p}{\rho} - \|\partial S\|(\ball{p}{\rho}) \right).
\end{aligned}   
\end{equation}  }

We denote by \((\cdot)^{\perp}\) the projection onto the approximate tangent plane to \(T\) at \(p\). Thus, using classical theory of slicing of currents, we have for a.e. \(0 \leq s<t\) that
\begin{equation*}
\int_{s}^{t} \|\partial S\|(\ball{p}{\rho})\diff \rho=\int_{\ball{p}{t} \setminus \ball{p}{s}} \frac{\left|{(x-p)}^{{T}}\right|}{|{(x-p)}|} d\|T\|.
\end{equation*}

Integrating \eqref{differential-equation2} on \((s,t)\subset (0,r_1)\), {using the last displayed inequality,} and applying \cite[Lemma 2.2]{hirsch2019uniqueness} (their proof works straightforwardly for \(f\), as in their notation, absolutely continuous, we just need to recall that \(\omega\) is absolutely continuous), we conclude
\begin{equation*}
\begin{aligned}
e^{m D(t)} \frac{\|T\|\left(\ball{p}{t}\right)}{t^{m}}-e^{m D(s)} \frac{\|T\|\left(\ball{p}{s}\right)}{s^{m}} & \geq \int_{\ball{p}{t} \setminus \ball{p}{s}} \frac{e^{mD(|z-p|)}}{|z-p|^{m}}\left(1-\frac{\left|(z-p)^{\perp}\right|}{|z-p|}\right) d\|T\|(z) \\
& \geq \int_{\ball{p}{t} \setminus \ball{p}{s}} e^{mD(|z-p|)} \frac{\left|(z-p)^{\perp}\right|^{2}}{2|z-p|^{m+2}} d\|T\|(z).
\end{aligned}
\end{equation*}
\end{proof}

A very useful result often implicitly used in our theory is the following. To enunciate and prove it, we fix the notation of the flat distance between two \(m\)-dimensional {integral currents \(T\) and \(S\), i.e., \(T, S\in\textbf{I}_m(U)\)}, \(U\) open and \(A\cc U\) as follows:
\begin{equation*}\mathcal{F}_{A}(T,S) = \infimo{\|R\|(A)+\|\tilde{T}\|(A)}{T-S=R+\partial\tilde{T} \ \text{with} \ R\in\textbf{I}_m(U) \ \text{and} \ \tilde{T}\in\textbf{I}_{m+1}(U)}.\end{equation*}

\begin{lemma}[Sequences of \(\omega_k\)-almost minimizers]\label{L:sequences-almost-minimal}
For each \(k\in\N\), let \(U\subset\R^{m+n}\) be an open set, we assume that
\begin{enumerate}[\upshape (a)]
    \item
    \(T_k\in\mathbf{I}_{m}(U)\) is \(\omega_k\)-almost area minimizing currents in \(U\),

    \item {$D_k(\rho):= \int_0^{\rho}\frac{\omega_k(\rho)}{\rho}\mathrm{d}\rho$ is such that $\sup_k (D_k(s_{\omega_k})) <+\infty$,}
    
    \item\label{seq_alm_min:b}  \(\partial T_k = 0\) for each \(k\in\N\),
    
    \item\label{seq_alm_min:c}  \(\limsup_{k\to+\infty}\|T_k\|(U)<+\infty,\)

    \item\label{seq_alm_min:d}  \(({r_1})_k\) from Proposition \ref{P:almost-monotonicity-times} satisfies \(r_0 := \liminf_{k\to +\infty}({r_1})_k>0,\)    
    
    \item\label{seq_alm_min:e}  \(\omega := \limsup_{k\to+\infty}\omega_k\) satisfies the assumptions of Definition \ref{D}, i.e., its domain contains \((0,r_0),\) it is absolutely continuous and \(\omega = o(1).\)
\end{enumerate} 
Then{, up to subsequences,} we have that there exists \(T\in\mathbf{I}_m(U)\) such that 
\begin{enumerate}[\upshape (i)]
    \item\label{seq_alm_min:i}  \(T_k \rightharpoonup T\),
    
    \item\label{seq_alm_min:ii}  \(T\) is \(\omega\)-almost minimizing in \(U,\)
    
    \item\label{seq_alm_min:iii}  \(\|T\|\leq \liminf_{k\to+\infty}\|T_k\|\leq\limsup_{k\to+\infty}\|T_k\| \leq \left(1 + \sup\omega \right) \|T\|,\)
    
    \item\label{seq_alm_min:iv}  \(\mathcal{F}(T_k,T)\to 0\) as \(k\) goes to \(+\infty\),
    
    \item\label{seq_alm_min:v}  \(\spt(T_k)\) converges in the Hausdorff distance sense to \(\spt(T)\).
\end{enumerate}
\end{lemma}
\begin{proof}
Thanks to \eqref{seq_alm_min:b} and \eqref{seq_alm_min:c}, we can use standard compactness results (one can consult \cite[Section 4.2]{Fed}) to ensure the existence of \(T\in\textbf{I}_m(U)\) such that \(T_{k} \rightharpoonup T\), up to a subsequence, in the sense of currents so \eqref{seq_alm_min:i} is proved. The equivalence between \eqref{seq_alm_min:i} and \eqref{seq_alm_min:iv} is given by \cite[Theorem 7.2]{simon2014introduction} again using \eqref{seq_alm_min:b} and \eqref{seq_alm_min:c}. Now we want to prove \eqref{seq_alm_min:ii} and \eqref{seq_alm_min:iii}. Let us write \(T_{k} - T = R_{k} + \partial \tilde{T}_{k}\) in \(\ball{p}{R+2}, p\in U, R\in (0,{s_{\omega}})\) with 
\begin{equation*}{\lim\sup}_{k\to+\infty}\biggl(\|R_{k}\| (\ball{p}{R+1}) + \|\tilde{T}_{k}\| (\ball{p}{R+1})\biggr) = 0.\end{equation*}

Thus, since all the measures involved are Radon measures, for almost every \(s \in (R, R+1)\), it follows that
\begin{equation}\label{zero-mass-at-infinity}
{\lim\sup}_{k\to+\infty}\| R_{k}\| (\ball{0}{s}) = 0
\end{equation}  
and 
\begin{equation}\label{zeromassofslice}
{\lim\sup}_{k\to+\infty}\massnorm{\slice{\tilde{T}_{k}}{d}{s}} = 0.
\end{equation}

Note that (\ref{zeromassofslice}) follows directly from the formula of the slice and the fact that \(T_{k}\) converges to \(T\) in the sense of currents. We use again the slice formula to get
\begin{equation}\label{decomposition-of-blowup}
T_{k}\res {\ball{p}{s}} = T\res {\ball{p}{s}} + R_{k}\res {\ball{p}{s}} - \slice{\tilde{T}_{k}}{d}{s} + \partial (\tilde{T}_{k}\res {\ball{p}{s}}).
\end{equation}

The \(\omega_k\)-almost minimality condition gives
\begin{equation*} \| T_{k} \|(\ball{p}{s}) \leq \left(1+\omega_k(s)\right)\|T_{k}+\partial \tilde{T}_{k}\|\left(\ball{p}{s}\right).\end{equation*}

Putting into account the latter inequality, the triangle inequality and (\ref{decomposition-of-blowup}), we obtain that
\begin{equation*} \| T_{k} \|(\ball{p}{s}) \leq \left(1+\omega_k(s)\right)\biggl( \|T\|(\ball{p}{s}) + \|R_{k}\|(\ball{p}{s}) + \massnorm{\slice{\tilde{T}_{k}}{d}{s}} + 2\| \partial \tilde{T}_{k} \|(\ball{p}{s}) \biggr).\end{equation*} 

Note that, by construction, it follows that \(\| \partial \tilde{T}_{k} \|(\ball{p}{s}) \to 0\) as \(k\to+\infty\). Finally, by the lower semicontinuity of the mass, (\ref{zero-mass-at-infinity}), (\ref{zeromassofslice}) and the last equation passed through \(\limsup_{k\to+\infty}\), we conclude that
\begin{equation}\label{almost-convergence-supports}
\begin{aligned}
    \|T\|(\ball{p}{s})\leq \liminf_{k\to+\infty}\|T_k\|(\ball{p}{s})&\leq\limsup_{k\to+\infty}\|T_k\|(\ball{p}{s}) \\
    &\leq \left(1 + \limsup_{k\to+\infty}\omega_k(s) \right) \|T\|(\ball{p}{s}),
\end{aligned}
\end{equation} 
for any \(p\in U\) and a.e. \(s\in (0,R)\), which ensures \eqref{seq_alm_min:ii} and \eqref{seq_alm_min:iii}.

We proceed with the proof of \eqref{seq_alm_min:v} using a contradiction argument. We take \(K\subset\R^{m+n}\) a compact subset and assume that there exists \(\eta_0>0\) and \(q_k\in K\cap\spt(T_k)\) with \(\diff(q_k, \spt(T))>\eta_0>0\) for \(k\) sufficiently large. Since \(K\) is compact, up to a subsequence, we denote by \(q_0\) the limit of \((q_k)_{k\in\N}\). Clearly, we have that \(\diff_H(q_0, \spt(T))\geq\eta_0>0\). Hence, by convergence of Radon measures, i.e., \eqref{almost-convergence-supports} we have that
\begin{equation}\label{mass-zero}
0 = \|T\|\left(\ball{q_0}{\frac{3\eta_0}{4}}\right) \geq \limsup_{k\to+\infty}\|T_k\|\left(\ball{q_0}{\frac{\eta_0}{2}}\right).     
\end{equation} 

Provided \(k\) is sufficiently large, it is easy to see that \(\ball{q_k}{\frac{\eta_0}{8}} \subset \ball{q_0}{\frac{\eta_0}{4}}\) which in turn, by the \(\omega_k\)-almost monotonicity formula, Proposition \ref{P:almost-monotonicity-times}, implies that
\begin{equation}\label{contradiction-monotonicity}
\begin{aligned}
\|T_k\|\left(\ball{q_0}{\frac{\eta_0}{2}}\right) \overset{\eqref{almost-monotonicity-times}}&{\geq}  e^{m({D_k}(\sfrac{\eta_0}{4}) - {D_k}(\sfrac{\eta_0}{2}))}2^m\|T_k\|\left(\ball{q_k}{\frac{\eta_0}{8}}\right)\\
&\geq C(m,k,\eta_0)e^{m({D_k}(\sfrac{\eta_0}{4}) - {D_k}(\sfrac{\eta_0}{2}))}, 
\end{aligned}
\end{equation} 
where \(C(m,k,\eta_0)>C>0\) where $C>0$ is a positive constant independent of $k$ because \(q_k\in\spt(T_k)\) and $T_k$ is an integer current and so the density of $T_k$ in $q_k$ is always a positive integer number greater or equal to $1$. Notice that 
\begin{equation}\label{constant-positive}
\begin{aligned}
C(m,k,\eta_0)e^{m({D_k}(\sfrac{\eta_0}{4}) - {D_k}(\sfrac{\eta_0}{2}))} &\geq \frac{C(m,k,\eta_0)}{e^{m{D_k}(\sfrac{\eta_0}{2})}} \geq  \frac{C(m,k,\eta_0)}{e^{m{D_k}(\sfrac{\eta_0}{2})}} \\
&\geq \frac{C(m,k,\eta_0)}{e^{m\limsup_{k\to+\infty}\max_{s\in [0,{r_1}_k]}\left({D_k}(s)\right)}} \overset{(*)}{>} 0,
\end{aligned}    
\end{equation}
recall that \({r_1}_k\) from the \(\omega_k\)-almost monotonicity formula depends on \(m,n\) and \(\omega_k\), which ensures \((*)\) thanks to \eqref{seq_alm_min:d} and \eqref{seq_alm_min:e}. This finishes the proof of the theorem, since \eqref{mass-zero}, \eqref{contradiction-monotonicity} and \eqref{constant-positive} are in contradiction.
\end{proof}

\section{Almgren's stratification for \texorpdfstring{\(\omega\)}{Lg}-almost area minimizing currents}

The stratification process allows us to estimate the Hausdorff dimension of the set of points at which the current becomes infinitesimally more flat.  Although this seems to be a good measurement of regularity, since it mimics the behavior of smooth manifolds, the fact that the density and codimension are arbitrary makes the 'becoming flat' property insufficient to derive regularity. The famous Federer's example \(\{(z,w)\in\mathbb{C}^2: z^2 = w^3\}\) confirms it at the origin. 

Let us define regular and singular interior points.

\begin{defi}
Let \(T\) be an \(\omega\)-almost area minimizing integral current in \(\mathbb{R}^{m+n}\), we define the set of regular points as follows
\begin{equation*}
    \reg (T):= \{p\in\spt(T): \spt(T)\cap\ball{p}{r} \text{ is a } C^{1,\alpha} \text{ submanifold of } \R^{m+n} \text{ for some } \alpha,r>0\},
\end{equation*}
as well the set of singular points is defined as \(\sing(T):=\spt(T)\setminus\left(\spt(\partial T)\cup\reg(T)\right)\).
\end{defi}

Let us prove an important property for the density function and the existence of blowup limits. To that end, fix the notation \(\iota_{p,r}(x):= r^{-1}(x-p)\).

\begin{lemma}[Minimal tangent cones and density's upper semi-continuity]\label{L:UpperDens-TangCones}
Let \(T\) be an \(\omega\)-almost area minimizing integral current in \(\mathbb{R}^{m+n}\), $\int_0^{s_\omega}\frac{\omega(\rho)}{\rho}\mathrm{d}\rho < +\infty$, \(\partial T = 0\) and \(p\in \spt(T)\). Then \(\Theta^m(T,p)\geq 1\) and
\begin{enumerate}[\upshape (i)]
    \item \(\Theta^m(T,q)\) exists everywhere and is an upper semi-continuous functions of \(q\in\R^{m+n}\);

    \item For each sequence \(r_k\to 0\), there is a subsequence \(\{r_{k^{\prime}}\}\) such that \((\iota_{p,r_{k^{\prime}}})_{\#} T =: T_{p.r_{k^{\prime}}} \weak C\), where \(C\) is an integer \emph{area minimizing cone} in \(\R^{m+n}\) with \(\Theta^m(T,p) = \Theta^m(C,0)\).
\end{enumerate}
\end{lemma}
\begin{rem}
The uniqueness of \(C\) is a long-standing open problem in the literature, i.e., as we change the sequence \(\{r_k\}_{k\in\N}\) we may end up with different limits. At this point, the conical property and the minimality are the takeaways.
\qed\end{rem}
\begin{proof}
The existence everywhere of the density is a direct consequence of the \(\omega\)-almost monotonicity formula (Proposition \ref{P:almost-monotonicity-times}), as well the upper semi-continuity of the density with respect to the point.  

It remains to prove the existence of a limit of the blowup sequence \(\{T_{p.r_{k^{\prime}}}\}_{k^{\prime}\in\N}\) and the conical property of this limit. 

The existence of \(C\) follows from Lemma \ref{L:sequences-almost-minimal} which also gives the minimality of \(C\). Indeed, since \(\omega_{k^{\prime}}(r) = \omega(r_{k^{\prime}}r)\) is the minimality error of \(T_{p,r_{k^{\prime}}}\), we have that \(\lim_{k^{\prime}}\omega_{r_{k^{\prime}}} \equiv 0\) hence \(C\) is an integer area minimizing current. 

We prove the conical property as follows. Since we are dealing with Radon measures, for almost every \(\rho>0\), we have
\begin{equation*}
\begin{aligned}
    \| T_{p,r_{k^{\prime}}}\|(\ball{0}{\rho}) \to \|C\|(\ball{0}{\rho}) \Rightarrow \\
    \Rightarrow \frac{\|C\|(\ball{0}{\rho})}{\omega_m\rho^m} = \lim_{k^{\prime}\to\infty} \frac{\| T_{p,r_{k^{\prime}}}\|(\ball{0}{\rho})}{\omega_m\rho^m} = \lim_{k^{\prime}\to\infty}\frac{\| T\|(\ball{p}{r_{k^{\prime}}\rho})}{\omega_m \rho^m r_{k^{\prime}}^m} = \Theta^m(T,p).
\end{aligned}
\end{equation*}

The last inequality stated that the mass ration of \(C\) is equal to a fixed constant, namely \(\Theta^m(T,p)\), for almost every \(\rho\). Such fact implies that the left-hand side of the \(\omega\)-almost monotonicity formula (Proposition \ref{P:almost-monotonicity-times}) vanishes. Combined with \cite[Lemma 2.40]{simon2014introduction}, we obtain that \(C\) is a cone.
\end{proof}

The blowup limits above (\(C\)'s of Lemma \ref{L:UpperDens-TangCones}) are called \textbf{tangent cones to \(T\) at \(p\)}. Whenever it occurs that \(\spt(C)\) is a \(m\)-dimensional subspace of \(\R^{m+n}\), \(C\) will be named a \textbf{tangent plane to \(T\) at \(p\)}.

Even though, Almgren's stratification theorem gives an estimate, it is \textbf{not} an estimate for the singular set, since existence of flat tangent planes does \textbf{not} imply regularity in the arbitrary density and codimension setting. 

\begin{thm}[Almgren's stratification theorem for \(\omega\)-almost area minimizing currents]\label{T:Almgren-stratification}
Let \(T\) be an \(\omega\)-almost area minimizing integral current in \(\mathbb{R}^{m+n}\), $\int_0^{s_\omega}\frac{\omega(\rho)}{\rho}\mathrm{d}\rho < +\infty$, and \(\partial T = 0\). Then, for any \(\eta>0\), there exists a \emph{tangent plane} to \(T\) at \(p\) for \(\cH^{m-2+\eta}\)-a.e. \(p\in\spt(T)\).
\end{thm}
\begin{proof}
As remarked, the proof is an adaptation of \cite[Theorem 3.3, Chapter 7]{simon2014introduction}. Recalling that \cite[Theorem 3.3, Chapter 7]{simon2014introduction} is just an application of \cite[Theorem A.4]{simon2014introduction}, we will as well make a list of definitions to be able to apply \cite[Theorem A.4]{simon2014introduction} (since \cite[Theorem A.4]{simon2014introduction} does not require any minimality condition). 

We now define the set \(\mathscr{F}\subset \mathbf{I}_m(\R^{m+n})\) as follows
\begin{equation*}
    \mathscr{F}:=\left \{ S: S = \lim_{i\to +\infty } (\iota_{x_i,\lambda_i})_{\#}T, \{x_i\}_{i\in\N} \text{ is a converging sequence}, \text{ and } \lambda_i \to 0 \right\}.
\end{equation*}

By compactness, i.e., using Lemma \ref{L:sequences-almost-minimal}, we have that any \(S\in\mathscr{F}\) is an integer \emph{area minimizing} \(m\)-current. It is also straightforward to verify that
\begin{equation*}
    (\iota_{q,\lambda})_{\#}\mathscr{F} = \mathscr{F} , \quad \forall q\in\R^{m+n}, \quad \lambda >0.    
\end{equation*}

We define, for any \(A\in 2^{\R^{m+n}}\), \(L\) an \(m\)-dimensional subspace of \(\R^{m+n}\), \(\lambda>0\), and \(q\in\R^{m+n}\), the height function as follows:
\begin{equation*}
    h( A, L, \lambda, q) = \sup_{y\in A\cap\ball{q}{\lambda}}|\bp_{L^{\perp}}(y-q)|.
\end{equation*}

For any \(S\in\mathscr{F}\) and \(\beta>0\), we set
\begin{equation*}
    \mathscr{T}_{\beta}(S):= \{q\in\spt(S): h(\spt(S),L,\lambda, q)\leq \beta \lambda \text{ for some } \lambda>0 \text{ and an } m\text{-plane } L\subset\R^{m+n}\}.
\end{equation*}

It is straightforward to see that
\begin{equation*}
    (\iota_{q,\lambda})_{\#}\mathscr{T}_{\beta}(S) = \mathscr{T}_{\beta}((\iota_{q,\lambda})_{\#}S) , \quad \forall q\in\R^{m+n}, \quad \lambda >0. 
\end{equation*}

Using the \(\omega\)-almost monotonicity formula (Proposition \ref{P:almost-monotonicity-times}), we readily check that
\begin{equation*}
    \begin{cases} 
    \{S_j\}_{j\in\N}\subset\mathscr{T}, S_j \weak S \in\mathscr{F} \\
    \{p_j\}_{j\in\N}\subset \spt(S_j), p_j \to p\in\mathscr{T}_{\beta}(S)
    \end{cases} \Rightarrow \quad p_j \in \spt(S_j).
\end{equation*}

Denoting \(N= \begin{pmatrix} m+n \\ n\end{pmatrix}\) and, for each \(S\in\mathscr{F}\), define \(\upvarphi_S^{0}(q):=\Theta^m(S,q)\) and \(\upvarphi_S^{k}(q):=\Theta^m(S,q)N^k_S(q)\) for \(k\in\{1,\ldots,N\}\) where \(N_S^k\) is the \(k\)-th component of the orientation \(\vec{S}(q)\). We also set 
\begin{equation*}
    \mathscr{F}: \{\upvarphi_S: S\in\mathscr{T}\} \text{ and }\mathrm{sing}\upvarphi_S := \spt \Theta^m(S,\cdot)\setminus \mathscr{T}_{\beta}(S).
\end{equation*}

By \cite[Theorem A.4 (\ddag)]{simon2014introduction} (the notations and definitions above matches the same notations and definition prior to Theorem A.4), we obtain the existence of \(d\in\N\cap[0,m-1]\) such that \(\mathrm{dim}_{H}(\mathrm{sing}\upvarphi_S)\leq d\) for each \(S\in\mathscr{F}\). In another words, we obtain for any \(\eta>0\)
\begin{equation*}
    0=\cH^{d+\alpha}(\mathrm{sing}\upvarphi_S) = \cH^{d+\alpha}(\spt(S)\setminus \mathscr{T}_{\beta}(S)), \quad \forall \beta>0.
\end{equation*}

Taking a sequence \(\beta_j\to 0\), we obtain from the last equality
\begin{equation*}
    0=\cH^{d+\alpha}(\spt(S)\setminus\cup_{j\in\N} \mathscr{T}_{\beta_j}(S)).
\end{equation*}

Furthermore, it is easy to see, by the definition of \(h\), that it holds
\begin{equation*}
    q\in \cup_{j\in\N} \mathscr{T}_{\beta_j}(S) \Leftrightarrow T \text{ has a tangent plane at }q.
\end{equation*}

Finally, by the last two displayed equations, it remains to prove that \(d\leq m-2\). We argue by contradiction, assume that \(d>m-2\) which implies, since \(d\) is an integer, \(d=m-1\). Also from \cite[Theorem A.4 (\ddag\ddag)]{simon2014introduction}, we also get the existence of \(S\in\mathscr{T}\) and an \(d\)-dimensional plane \(L\) such that \(\mathrm{sing}\upvarphi_S = L\). Since \(S\) is a minimizing cone without boundary, it is well known (see for instance \cite[Lemma 3.5, Chapter 7]{simon2014introduction}) that it splits into \(\a{\R^{m-1}}\times S_0 \) where \(S_0\) is an \(1\)-dimensional minimizing cone. Since \(S_0\) has no boundary, it has to be a line, i.e., \(S_0=k\a{\ell}\). Such a fact, surely implies that \(S = \a{\R^{m-1}}\times S_0 = \a{\R^{m-1}}\times k\a{\ell}\) is flat, hence \(\mathrm{sing}\upvarphi_S = \emptyset\) which is a contradiction. Then \(d\leq m-2\) and we are done.
\end{proof}

\section{Strong Lipschitz approximation}

The goal of this section is to prove our main theorem regarding the Lipschitz approximation, i.e., Theorem \ref{T:strong-approx}, in which we provide superlinear estimates. 

We now define the excess which is one of the most important concepts in the regularity theory of currents, since it measures the deviation of the current with respect to an \(m\)-plane.

\begin{defi}[Excess measure]
Let \(T\in\mathbf{I}_m(\R^{m+n})\) satisfying Assumption \ref{A:LA}. We define the \textbf{excess measure} as
\begin{equation*} \be_T(A) := \|T\|(A\times\R^n) - Q\cH^m(A), \quad \forall A\subset\baseball{\bp(p)}{r}, \end{equation*}
and the \textbf{cylindrical excess} as
\begin{equation*} \bE(T,\cyl{p}{s}) = \frac{\be_T(\baseball{\bp(p)}{s})}{\omega_m s^m}, \quad \forall s\in (0,r),\end{equation*}
where from now on we denote \(\omega_m := \cH^m(\baseball{0}{1})\).
\end{defi}

In what follows, we will work under two assumptions which are the constancy assumption \ref{A:constancy} and the no boundary assumption \ref{A:no-boundary} described below.

\begin{assump}\label{A:LA}
There exist a point \(p\in\spt(T)\setminus\spt(\partial T)\), a radius \(r>0\) and an integer \(Q\in\N\setminus\{0\}\) such that 
\begin{equation}\label{A:constancy}\eqname{(CA)} 
\left(\bp_\# T\right)\res\cyl{p}{4r} = \Theta^m(\bp_\# T,p)\a{\baseball{\bp(p)}{4r}} := Q\a{\baseball{\bp(p)}{4r}},
\end{equation}
\begin{equation}\label{A:no-boundary}\eqname{(NB)} 
\partial T \res\cyl{p}{4r} = 0.
\end{equation}
\end{assump}

The assumption above is not restrictive when the strong Lusin type Lipschitz approximation is applied in the regularity theory, because of the following lemma which is the analogous of \cite[Lemma 1.6]{DS4}.

\begin{lemma}\label{L:preliminaryLA}
Whenever \(T\) is an \(\omega\)-almost area minimizing integral current in \(\R^{m+n}\), {$\int_0^{s_\omega}\frac{\omega(\rho)}{\rho}\mathrm{d}\rho < +\infty$}, \(p\in\spt(T)\setminus\spt(\partial T)\), there exists a geometric constant \(\eta=\eta(m,n,Q)>0\), such that, if
\begin{align}
    \Theta^m(T, p) &= Q\label{constancy0}\\
    \frac{\|T\|(\ball{p}{4r})}{\omega_m (4r)^m} &\leq Q + \eta,\label{constancy2}\\
    \max\{\omega({s_{\omega}}), \bE(T, \ball{p}{4r}) &:= \bE(T, \ball{p}{4r}, \pi_0)\} < \eta, \label{constancy3}
\end{align}
then we have that \((\bp_{\pi_0})_{\#}T\res\ball{p}{4r} = Q\a{\baseball{\bp(p)}{4r}}\).
\end{lemma}
\begin{proof}
Assume by contradiction that we have a sequence \(T_k\) of \(\omega_k\)-almost area minimizing integral currents with \(\omega_k\) converging uniformly to \(f\equiv 0\) as \(k\to\infty\), and a sequence of real numbers \(\eta_k\to 0\) as \(k\) goes to \(+\infty\) both satisfying \eqref{constancy2} and \eqref{constancy3} such that 
\begin{equation*} (\bp_{\pi_0})_{\#}T_k\res\ball{p}{4r} \neq Q\a{\baseball{\bp(p)}{4r}}, \forall k\in \N.\end{equation*}

Since \(\| T_k \|(\ball{p}{4r}) + \|\partial T_k\|(\ball{p}{4r})\) is uniformly bounded w.r.t. \(k\), we can then apply the Compactness Theorem for integral currents which gives the existence of an integral current \(T_\infty\in\mathbf{I}_m(\ball{p}{4r})\) such that \(T_k\rightharpoonup T_\infty\) in the sense of currents. By the \(\omega\)-almost monotonicity formula, Proposition \ref{monotonicity-formula}, we have that the convergence in the sense of currents implies the convergence of the measures and the Hausdorff convergence of the supports, Lemma \ref{L:sequences-almost-minimal}. Hence, by \eqref{constancy3}, we directly obtain that 
\begin{equation*} \bE(T_\infty, \ball{p}{4r}, \pi_0) = 0, \end{equation*}
which leads to \(\spt(T_\infty)\subset \ball{p}{4r}\cap\pi_0\) and then \(T_\infty = Q_\infty\a{\baseball{\bp(p)}{4r}}\) for some integer \(Q_\infty\). Using the convergence of \(\|T_k\| \rightharpoonup \|T_\infty\|\), \eqref{constancy2} and \eqref{constancy0}, we obtain that \(Q_\infty = Q\). Now, since \(T_k\) has no boundary in \(\ball{p}{r}\), i.e., \(\partial T_k\res\ball{p}{r} = 0\), for \(k\) large enough, we obtain by the Constancy Lemma (\cite[4.1.17]{Fed}) that 
\begin{equation*} (\bp_{\pi_0})_{\#}T_k \res\ball{p}{4r} = Q_k\a{\baseball{\bp(p)}{4r}}. \end{equation*}

Since we have the convergence of \(T_k\) to \(T_\infty\) in the sense of currents, we obtain that \(Q_k = Q\), for \(k\) large enough, which gives the contradiction. 
\end{proof}

\subsection{Weak Lipschitz approximation}

The weak Lipschitz approximation does not need any minimality condition to be proven, then we refer the reader directly to \cite[Proposition 2.2]{DS3}.

\begin{prop}[Weak Lusin type Lipschitz approximation]\label{P:weak-lip}
There exist two positive geometric constants \(\varepsilon_{wl} = \varepsilon_{wl}(m,n,Q)> 0 \) and \(C_{wl} = C_{wl}(m,n,Q)> 0 \) such that, if 
\begin{enumerate}[\upshape (a)]
    \item \(T\) satisfies Assumption \ref{A:LA},
    
    \item \(\bE(T,\cyl{p}{4r})<\varepsilon_{wl}\),
    
    \item \(\beta\in (0,\frac{1}{2m}]\),
\end{enumerate}
hold, we have that exist a set \(K\subset\baseball{\bp(p)}{\frac{7r}{2}}\) and a Lipschitz function \(f:\baseball{\bp(p)}{\frac{7r}{2}} \to \Is{Q}\) which satisfies
\begin{align}
\operatorname{Lip}(f) &\leq C_{wl}\bE(T,\cyl{p}{4r})^{\beta},\label{E:weak-app:lip-bound} \\
\mathbf{G}_f\res\left(K\times \R^n\right) &= T\res \left(K\times \R^n\right),\label{E:weak-app:=inK}\\
\| T-\mathbf{G}_f \|\left( \baseball{\bp(p)}{\frac{7r}{2}}\setminus K \right) &\leq C_{wl} r^m \bE(T,\cyl{p}{4r})^{1-2\beta},\nonumber \\
\cH^m\left(\baseball{\bp(p)}{\frac{7r}{2}}\setminus K \right) &\leq C_{wl} r^m \bE(T,\cyl{p}{4r})^{1-2\beta}. \nonumber 
\end{align}
\end{prop}

Let us now enunciate the analogous of \cite[Lemma 2.2]{DSS2} with our more general almost minimality condition.

\begin{lemma}[Homotopy Lemma]\label{L:homotopy}
Let \(T\) be \(\omega\)-almost area minimizing which satisfies Assumption \ref{A:LA}, {$\int_0^{s_\omega}\frac{\omega(\rho)}{\rho}\mathrm{d}\rho < +\infty$}, and \(3r<{s_{\omega}}\). There are positive geometric constants \(\varepsilon_h=\varepsilon_h(m,n,Q)>0\) and \(C_h=C_h(m,n,Q)>0\) such that, whenever \(\bE(T,\cyl{p}{4r})<\varepsilon_h<\varepsilon_{wl}\), it holds
\begin{equation*} \|T\|(\cyl{p}{3r})\leq \|R\|(\cyl{p}{3r}) + C_h\omega (3r) r^{m}\bE(T,\cyl{p}{4r})^{\frac{1}{2}},\end{equation*}
for every \(R\in\mathbf{I}_m(\cyl{p}{3r})\) with \(\partial R = \partial \left(T\res\cyl{p}{3r}\right)\).
\end{lemma}
\begin{proof}
In the proof of \cite[Lemma 2.2]{DSS2}, the authors do not use any minimality property to show the existence of an \((m+1)\)-current \(S^{\prime\prime}\in\mathbf{I}_{m+1}(\cyl{p}{3r})\) (notice that \(S^{\prime\prime}\) is the same notation that the authors use for such current) such that
\begin{align}
\partial S^{\prime\prime} &= \left(T-R\right)\res\cyl{p}{3r}, \label{L:homot:1} \\
\| S^{\prime\prime} \|(\cyl{p}{3r}) &\leq C_0r^{m+1}\bE(T,\cyl{p}{3r})^{\frac{1}{2}}. \label{L:homot:2}
\end{align}

Then, applying our \(\omega\)-minimality condition on \(T\), we obtain that 
\begin{equation*}
\begin{aligned}
\|T\|(\cyl{p}{3r}) &\leq \left(1+\omega(r)\right)\|T-\partial S^{\prime\prime}\|(\cyl{p}{3r}) \\
\overset{\eqref{L:homot:1}}&{\leq} \|R\|(\cyl{p}{3r}) + \omega(r)\|T-\partial S^{\prime\prime}\|(\cyl{p}{3r}) \\
\overset{(\ast)}&{\leq} \|R\|(\cyl{p}{3r}) + C_1\frac{\omega(r)}{r}r^{m+1} \\
\overset{(\ast)}&{\leq} \|R\|(\cyl{p}{3r}) + C_2\frac{\omega(r)}{r}\|S^{\prime\prime}\|(\cyl{p}{3r}) \\
\overset{\eqref{L:homot:2}}&{\leq} \|R\|(\cyl{p}{3r}) + C_3\omega(r)r^{m}\bE(T,\cyl{p}{3r})^{\frac{1}{2}},
\end{aligned}
\end{equation*}
where in \((\ast)\) we use standard density estimates, see for instance \cite[Section 4.1.28]{Fed}.
\end{proof}

\begin{defi}[Barycenter of a \(Q\)-tuple]
We define \(\etaa(P) = \frac{1}{Q}\sum_{i=1}^Q P_i\) for every \(P = \sum_{i=1}^Q\a{P_i}\in\Is{Q}\).
\end{defi}

\begin{prop}[Approximation by minimizers of the Dirichlet energy]\label{P:harm-approx}
Given two positive real numbers \(\eta>0\) and \(\beta\in (0,\frac{1}{2m})\), there exists a positive geometric constant \(\varepsilon_{ha}=\varepsilon_{ha}(m,n,Q,\eta,\beta)>0\) such that, if
\begin{enumerate}[\upshape (a)]
    \item\label{harm_approx:(a)} \(T\) is \(\omega\)-almost area minimizing, {$\int_0^{s_\omega}\frac{\omega(\rho)}{\rho}\mathrm{d}\rho < +\infty$,} and satisfies Assumption \ref{A:LA},
    
    \item\label{harm_approx:(b)} \(\bE(T,\cyl{p}{4r})<\varepsilon_{ha}<\varepsilon_h\),
    
    \item\label{harm_approx:(c)} \(\omega(r) \leq \varepsilon_{ha}\bE(T,\cyl{p}{4r})^{\frac{1}{2}}\),
\end{enumerate}
hold, then there exists \(w:\baseball{\bp(p)}{\frac{7r}{2}} \to \Is{Q}\) which minimizes the Dirichlet energy and satisfies
\begin{align*}
\int_{\baseball{\bp(p)}{2r}}|Df|^2\diff\!\cH^m  &\leq \eta\be_T(\baseball{\bp(p)}{4r}),\\
\frac{1}{r^2}\int_{\baseball{\bp(p)}{2r}}\cG(f,g)^2\diff\!\cH^m  + \int_{\baseball{\bp(p)}{2r}}\left(|Df|-|Dw|\right)^2\diff\!\cH^m  &\leq \eta\be_T(\baseball{\bp(p)}{4r}),\\
\int_{\baseball{\bp(p)}{2r}}|D(\etaa\circ f) - D(\etaa\circ w)|^2\diff\!\cH^m  &\leq \eta\be_T(\baseball{\bp(p)}{4r}).
\end{align*}
where \(f\) is the weak Lusin type Lipschtiz approximation given by Proposition \ref{P:weak-lip}.
\end{prop}
\begin{proof}
Note that condition \eqref{harm_approx:(a)} and \eqref{harm_approx:(b)} put us in position to apply the Homotopy Lemma (Lemma \ref{L:homotopy}) and then the proof of \cite[Theorem 3.1]{DS3} follows straightforwardly.
\end{proof}
\begin{rem}
Note that the approximation above at interior points is weaker than the one we can get at boundary points. Indeed, for integral \(m\)-currents that minimize the area in \(\R^{m+n}\) and take the boundary with arbitrary boundary multiplicity, in \cite[Theorem 4.12]{nardulli2022density}, we proved that \(w\) is indeed multi-copies of a classical harmonic function. However, at interior points, we can prove that \(w\) is a minimizer of the Dirichlet energy but not necessarily \(Q\) copies of an harmonic function.
\qed\end{rem}

\begin{lemma}[Weak excess decay]\label{L:weak-excess-decay}
For every \(\eta>0\), there exist a positive geometric constants \(\varepsilon_{we}=\varepsilon_{we}(m,n,Q,\eta)>0\) and a positive constant \(C_{we}=C_{we}(m,n,Q,\eta)>0\), if 
\begin{enumerate}[\upshape (a)]
    \item\label{weak-decay:(a)} \(T\) is \(\omega\)-almost area minimizing under Assumption \ref{A:LA} and {$\int_0^{s_\omega}\frac{\omega(\rho)}{\rho}\mathrm{d}\rho < +\infty$},
    
    \item\label{weak-decay:(b)} \(\bE(T,\cyl{p}{4r})<\varepsilon_{we}<\varepsilon_{ha}\),
    
    \item\label{weak-decay:(c)} \(A\subset\baseball{\bp(p)}{r}\) Borel set with \(\cH^m(A) \leq \varepsilon_{we}\omega_mr^m\),
\end{enumerate}
hold, then we have that
\begin{equation}\label{E:weak-excess-decay}
\be_T(A) \leq \eta\be_T(\baseball{\bp(p)}{4r})r^m + C_{we}\omega(r)^2 r^{m}.
\end{equation}  
\end{lemma}
\begin{proof}
We first assume that \(\omega({s_{\omega}})\geq\varepsilon_{ha}\bE(T,\cyl{p}{4r})^{\frac{1}{2}}\) which leads to
\begin{equation*} \be_T(A) = \bE(T,\cyl{p}{4r})\omega_m(4r)^m \leq C_0\bE(T,\cyl{p}{4r}) \leq \varepsilon_{ha}^{-2}\omega^2({s_{\omega}}),\end{equation*}
which gives \eqref{E:weak-excess-decay}. On the other hand, if we have \(\omega({s_{\omega}})\leq\varepsilon_{ha}\bE(T,\cyl{p}{4r})^{\frac{1}{2}}\) which is \eqref{harm_approx:(c)} of Proposition \ref{P:harm-approx}, then we can apply it and the proof goes as in \cite[Proposition 3.2]{DSS2}.
\end{proof}

We fix the notation for the \textbf{density of the excess measure} as follows
\begin{equation}\label{D:density-excess-measure}
\bd_T(q) := \limsup_{s\to 0} \frac{\be_T(\baseball{\bp(q)}{s})}{\omega_m s^m}. 
\end{equation} 

\subsection{Higher integrability of the excess measure's density}

One crucial point in the theory which will allow us to improve our weak excess decay, Lemma \ref{L:weak-excess-decay}, is the higher integrability of the function \(\bd_T\), it means that \(\bd_T\in\mathrm{L}^p(\baseball{\bp(p)}{2r})\) for some \(p>1\). By the Taylor expansion, we can easily compare \(|Df|\) (from the harmonic approximation, Proposition \ref{P:harm-approx}) and \(\bd_T\), hence we can reduce the problem of studying the higher integrability of minimizers of the Dirichlet energy to study it for \(\bd_T\). It will be used to prove the strong excess decay, Theorem \ref{T:strong-excess-decay}. For a more detailed discussion about this topic we refer the reader to \cite{Sp12}, \cite[Section 6]{DS3} and \cite{Sp10}.

\begin{prop}[Higher integrability of the density of the excess measure]\label{P:higher-int}
There exist positive geometric constants \(a=a(m,n,Q)>1, \varepsilon_a=\varepsilon_a(m,n,Q)>0\) and \(C_a=C_a(m,n,Q)>0\) such that, if
\begin{enumerate}[\upshape (a)]
    \item \(T\) is \(\omega\)-almost area minimizing under Assumption \ref{A:LA} and {$\int_0^{s_\omega}\frac{\omega(\rho)}{\rho}\mathrm{d}\rho < +\infty$,}
    
    \item \(\bE(T,\cyl{p}{4r})<\varepsilon_a<\varepsilon_{we}\),
\end{enumerate}
hold, then
\begin{equation}\label{E:thm:higher-int}
\int_{\{\bd_T \leq 1\}\cap\baseball{\bp(p)}{2r}}\bd_T^{a}\diff\cH^m \leq C_a  \left[\bE(T,\cyl{p}{4r})^a + \omega(r)^2\bE(T,\cyl{p}{4r})^{a-1}\right]r^m.
\end{equation} 
\end{prop}
\begin{proof}
The inequality is trivial if we consider that \(\bE(T,\cyl{p}{4r})=0\), we then assume w.l.o.g. that \(\bE := \bE(T,\cyl{p}{4r}) >0\). We claim that 
\begin{itemize}
    \item There exist constants \(\gamma = \gamma(m,\varepsilon_{we})\geq 2^m\) and \(\theta = \theta(m,n,\varepsilon_{we})>0\) such that, for every \(c\in \left[1, \frac{1}{\gamma\bE}\right]\) and \(s\in\left[2r,4r(1-c^{-\frac{1}{m}})\right]\), we have 
\begin{equation}\label{T:higher-int:claim}
\begin{aligned}
    \int_{\{c\gamma\bE\leq\bd_T\leq 1\}\cap\baseball{\bp(p)}{s}} \bd_T\diff\cH^m &\leq \gamma^{-\theta}r^m\int_{\{\frac{c\bE}{\gamma}\leq\bd_T\leq1\}\cap\baseball{\bp(p)}{s+rc^{-\frac{1}{m}}}}\bd_T\diff\cH^m\\
    &\quad + C_{we}c^{-1}\omega(r)^2r^{m},
\end{aligned}
\end{equation}
\end{itemize}

We now show how to obtain the statement of the theorem, i.e., equation \eqref{E:thm:higher-int}. We want to apply the claim for \(c=\gamma^{2k}\), to that end, we need
\begin{equation*} \gamma^{2k}\leq \frac{1}{\gamma \bE} \Rightarrow k \leq \frac{1}{2}\left( \log_{\gamma}\bE^{-1} - 1\right). \end{equation*}

Henceforth we denote by \(k_0\) the biggest number in \(\N\) that satisfies the latter inequality, and we define \(s_1:=2r\) and \(s_k=s_{k-1} + r\gamma^{-\frac{2k}{m}}, \forall k\leq k_0\). Recall that \(s_k\) is increasing, thus, we may apply the claim for \(c=\gamma^{2k}\) and \(s = s_k\), 
\begin{equation*}\begin{aligned} 
    \int_{\{\gamma^{2k+1}\bE\leq\bd_T\leq 1\}\cap\baseball{\bp(p)}{s_k}} \bd_T\diff\cH^m \overset{\eqref{T:higher-int:claim}}&{\leq} \gamma^{-\theta}r^m\int_{\{\gamma^{2k-1}\bE\leq\bd_T\leq1\}\cap\baseball{\bp(p)}{s_{k+1}}}\bd_T\diff\cH^m\\
    &\quad + C_{we}\gamma^{-2k}\omega(r)^2 r^{m}, \quad \forall k\in\{1,\dots, k_0\}.\\
\end{aligned}\end{equation*}

So, if we iterate the last equation, it is then immediate to see that
\begin{equation}\label{E:higher-int:interation}
\begin{aligned} 
    \int_{\{\gamma^{2k+1}\bE\leq\bd_T\leq 1\}\cap\baseball{\bp(p)}{2r}} \bd_T\diff\cH^m \overset{\eqref{T:higher-int:claim}}&{\leq} \gamma^{-(k-1)\theta}r^m\int_{\{\gamma^{2k_0-1}\bE\leq\bd_T\leq1\}\cap\baseball{\bp(p)}{s_{k_0}}}\bd_T\diff\cH^m\\
    &\quad + C_{we}\omega(r)^2 r^{m}\sum_{i=0}^{k-2}\gamma^{-2(k-i)+i\theta}, \quad \forall k\in\{1,\dots, k_0\}.\\
\end{aligned}
\end{equation}

We thus fix any number \(a=a(m,n,Q)\in(1,1+\sfrac{\theta}{2})\), if necessary, we reduce \(\theta\) in order to have \(\theta < \sfrac{2}{m}\) and define the sets \(A_0:=\{\bd_T < \gamma\bE\}\cap\baseball{\bp(p)}{2r}, A_k:= \{\gamma^{2k-1}\bE\leq\bd_T\leq\gamma^{2k+1}\bE\}\cap\baseball{\bp(p)}{2r}, \forall k\in\{1,\dots,k_0\},\) and \(A_{k_0+1}:=\{\gamma^{2k_0+1}\bE\leq\bd_T\leq 1\}\cap\baseball{\bp(p)}{2r}\). Therefore, we obtain that
\begin{equation*}
\begin{aligned}
\int_{\{\bd_T\leq 1\}\cap\baseball{\bp(p)}{2r}}&\bd_T^{a}\diff\cH^m \leq \sum_{k=0}^{k_0+1}\int_{A_k}\bd_T^{a}\diff\cH^m\leq \sum_{k=0}^{k_0+1}\gamma^{(2k+1)(a-1)}\bE^{a-1}\int_{A_k}\bd_T\diff\cH^m\\
\overset{\eqref{E:higher-int:interation}}&{\leq} C_1\sum_{k=0}^{k_0+1}\left[\gamma^{k(2(a-1)-\theta)}r^m\bE^{a}+C_2\bE^{a-1}\omega(r)^2r^{m}\sum_{i=0}^{k-2}\gamma^{k(2(a-1)-\frac{2}{m})-i(\frac{2}{m}-\theta)}\right]\\
&\leq C_3r^m\bE^{a} + C_3\bE^{a-1}\omega(r)^2r^{m}\sum_{k=0}^{k_0+1}\gamma^{k(2(a-1)-\theta)},
\end{aligned}
\end{equation*}
which concludes the proof of the theorem.

Let us prove the initial claim \eqref{T:higher-int:claim}. Let us fix the constant \(\eta = 2^{-2m-N}\) and the corresponding dimensional constant \(M=M(m,n)>0\) which is the constant given by the Besicovich's covering theorem, c.f. \cite[Section 1.5.2]{EG}, the natural number \(N\in\N\) such that \(M<2^{N-1}\) and
\begin{equation}\label{T:higher-int:def:gamma}
    \gamma =\gamma(m,\varepsilon_{we}) := \max\left\{ 2^m, \varepsilon_{we}^{-1} \right\}
\end{equation} 

So, as in the claim, take arbitrary numbers \(c\in \left[1, \frac{1}{\gamma\bE}\right]\) and \(s\in\left[2r,4r (1-c^{-\frac{1}{m}})\right]\).
We obtain the following inequality by classical rescaling of the excess
\begin{equation}\label{defining-r(p)}
\bE(T,\cyl{q}{t}) \leq \left(\frac{4r}{t}\right)^m\bE(T,\cyl{q}{4r}) \leq c\bE, \quad \forall t \geq 4r c^{-\frac{1}{m}}.
\end{equation}

Note that, by \eqref{defining-r(p)}, we can define
\begin{equation*} 4r(q):= \min\left \{ t\in \left[0, 4r c^{-\frac{1}{m}}\right]: \bE(T,\cyl{q}{t}) \leq c\bE \right \},\end{equation*}
it remains to show that \(r(q)>0\) and it follows for almost every \(q\in \{ c\gamma\bE\leq\bd_T\leq 1\}\cap \cyl{q}{s}\), since
\begin{equation*} \lim_{s\to 0}\bE(T,\cyl{q}{s}) \overset{\eqref{D:density-excess-measure}}{=} \bd_T(q) \geq c\gamma\bE \overset{\eqref{T:higher-int:def:gamma}}{\geq} 2^m c\bE. \end{equation*}

By the definition of \(r(q)\) and the fact that it is a positive constant, we obviously have the inequalities
\begin{align}
\bE(T,\cyl{q}{4r(q)}) \leq c\bE,\label{6.12-1}\\
\bE(T,\cyl{q}{t}) \geq c\bE, \quad \forall t \in \left( 0 , 4r(q) \right).\label{6.12-2}
\end{align}

Now, we aim to apply the weak excess decay, Lemma \ref{L:weak-excess-decay}, to the current \(T\res\cyl{q}{4r(q)}\) and the set \(A=\{ c\gamma\bE\leq \bd_T \}\cap\baseball{\bp(q)}{4r(q)}\). It is clear that we are under hypothesis \eqref{weak-decay:(a)} of Lemma \ref{L:weak-excess-decay}, let us check \eqref{weak-decay:(b)}:
\begin{equation*} \bE(T,\cyl{q}{4r(q)}) \overset{\eqref{6.12-1}}{\leq} c\bE \leq \frac{E}{\gamma E} = \gamma^{-1} \overset{\eqref{T:higher-int:def:gamma}}{<}  \varepsilon_{we}, \end{equation*}
and \eqref{weak-decay:(c)} follows from:
\begin{equation*}\begin{aligned}
\cH^m(A) \leq \frac{1}{c\gamma\bE}\int_{A}\bd_T\diff\cH^m(q) \overset{\textup{Fatou's Lemma}}&{\leq} \frac{1}{c\gamma\bE}\lim_{s\to 0}\int_{A}\bE(T,\cyl{q}{s})\diff\cH^m(q)\\
\overset{\eqref{6.12-1}}&{\leq} \frac{\cH^m(\baseball{\bp(q)}{4r(q)})}{\gamma}\\
\overset{\eqref{T:higher-int:def:gamma}}&{<} \varepsilon_{we}\omega_m(4r(q))^m.
\end{aligned}\end{equation*}

Now, recalling that \(\eta = 2^{-2m-N}\), we are able to apply the weak excess decay, Lemma \ref{L:weak-excess-decay}, and Lebesgue's differentiation theorem to derive
\begin{equation}\label{E:higher-int:dTinA}
\begin{aligned}
\int_{A}\bd_T\diff\cH^m \leq \be_T(A) \overset{\eqref{E:weak-excess-decay}}&{\leq}  2^{-2m-N}\be_T(\baseball{\bp(q)}{4r(q)})r(q)^m + C_{we} \omega(r(q))^{2} r(q)^m \\
&= 2^{-N}\bE(T,\cyl{q}{4r(q)}) + C_{we} \omega(r(q))^{2} r(q)^m \\
\overset{\eqref{6.12-1}}&{\leq} 2^{-N}c\bE + C_{we} \omega(r(q))^{2} r(q)^m \\
\overset{\eqref{6.12-2}}&{\leq} 2^{-N}\bE(T,\cyl{q}{r(q)}) + C_{we} \omega(r(q))^{2} r(q)^m \\
&= 2^{-N}\be_T(\baseball{\bp(q)}{r(q)})r(q)^m + C_{we} \omega(r(q))^{2} r(q)^m.
\end{aligned}
\end{equation}

From the latter chain of inequalities \eqref{E:higher-int:dTinA} and the fact that \(\{\bd_T > 1\}\cap\baseball{\bp(q)}{r(q)} \subset A\), we obtain
\begin{equation}\label{E:d>1}
\int_{\{\bd_T > 1\}\cap\baseball{\bp(q)}{r(q)} }\bd_T\diff\cH^m \leq \int_{A}\bd_T\diff\cH^m 
\leq 2^{-N}\be_T(\baseball{\bp(q)}{r(q)})r(q)^m + C_{we} \omega(r(q))^{2} r(q)^m.   
\end{equation}

It is easy to see that
\begin{equation}\label{E:d<ce-over-gamma}
\int_{\{ \bd_T < \frac{c\bE}{\gamma}\}\cap\baseball{\bp(q)}{r(q)}}\bd_T\diff\cH^m \leq \frac{c\bE}{\gamma}\omega_m r(q)^m \overset{\eqref{6.12-2}}{\leq} \gamma^{-1}\be_T(\baseball{\bp(q)}{r(q)}),
\end{equation}

Putting the Lebesgue differentiation theorem, \eqref{E:d>1} and \eqref{E:d<ce-over-gamma} into account, we infer that
\begin{equation*}
\begin{aligned}
    \be_T(\ball{\bp(q)}{r(q)}) &= \int_{\{ \bd_T < \frac{c\bE}{\gamma}\}\cap\baseball{\bp(q)}{r(q)}}\bd_T\diff\cH^m + \int_{\{\frac{c\bE}{\gamma} \leq \bd_T \leq 1\}\cap\baseball{\bp(q)}{r(q)} }\bd_T\diff\cH^m \\
    &\quad + \int_{\{\bd_T > 1\}\cap\baseball{\bp(q)}{r(q)} }\bd_T\diff\cH^m \\
    \overset{\eqref{E:d>1},\eqref{E:d<ce-over-gamma}}&{\leq} \left(2^{-N}r(q)^m+\gamma^{-1}\right)\be_T(\baseball{\bp(q)}{r(q)}) + C_{we} \omega(r(q))^{2} r(q)^m \\
    &\quad + \int_{\{\frac{c\bE}{\gamma} \leq \bd_T \leq 1\}\cap\baseball{\bp(q)}{r(q)} }\bd_T\diff\cH^m,
\end{aligned}
\end{equation*}
which surely implies
\begin{equation}\label{E:higher-int:excess-measure}
\left[ 1- 2^{-N}-\gamma^{-1} \right]\be_T(\ball{\bp(q)}{r(q)}) \leq \int_{\{\frac{c\bE}{\gamma} \leq \bd_T \leq 1\}\cap\baseball{\bp(q)}{r(q)} }\bd_T\diff\cH^m + C_{we} \omega(r(q))^{2} r(q)^m.
\end{equation}

Since \(\{c\gamma\bE\leq\bd_T\leq1\}\cap\baseball{\bp(q)}{r(q)} \subset A\), it is guaranteed that
\begin{equation}\label{E:higher-int:prefinal}
\begin{aligned}
    \int_{\{c\gamma\bE\leq\bd_T\leq1\}\cap\baseball{\bp(q)}{r(q)}}\bd_T\diff\cH^m &\leq \int_A\bd_T\diff\cH^m \\
    \overset{\eqref{E:higher-int:dTinA}}&{\leq} 2^{-N}\be_T(\baseball{\bp(q)}{r(q)})r(q)^m + C_{we} \omega(r(q))^{2} r(q)^m \\
    \overset{\eqref{E:higher-int:excess-measure}}&{\leq} \frac{2^{-N}}{1- 2^{-N}-\gamma^{-1}}r(q)^m\int_{\{\frac{c\bE}{\gamma}\leq\bd_T\leq1\}\cap\baseball{\bp(q)}{r(q)}}\bd_T\diff\cH^m  \\
    &\quad + C_{we} \omega(r(q))^{2} r(q)^m,
\end{aligned}
\end{equation}
using that \(\gamma \geq 2^m\), we can bound \(\frac{2^{-N}}{1- 2^{-N}-\gamma^{-1}}\) by \(2^{-N+1}\).
Recalling that \(M\) is given by the Besicovich's covering theorem, \cite[Section 1.5.2]{EG}, we choose \(M\) families \(\{\mathcal{A}_1,\dots,\mathcal{A}_M\}\) of closed disjoint balls \(\overline{\baseball{\bp(q)}{r(q)}}\) with center in \(\baseball{\bp(p)}{s}\) such that their union covers \(\{c\gamma\bE\leq\bd_T\leq 1\}\cap\baseball{\bp(p)}{s}\), then by \eqref{E:higher-int:prefinal} we have that
\begin{equation*}
\begin{aligned}
    \int_{\{c\gamma\bE\leq\bd_T\leq1\}\cap\baseball{\bp(p)}{s}}\bd_T\diff\cH^m &\leq \sum_{i=1}^M\sum_{q\in\mathcal{A}_i} \int_{\{c\gamma\bE\leq\bd_T\leq1\}\cap\baseball{\bp(q)}{r(q)}}\bd_T\diff\cH^m \\
    \overset{\eqref{E:higher-int:prefinal}}&{\leq} 2^{-N+1}\sum_{i=1}^M\sum_{q\in\mathcal{A}_i} \left[\int_{\{\frac{c\bE}{\gamma}\leq\bd_T\leq1\}\cap\baseball{\bp(q)}{r(q)}}\bd_T\diff\cH^m \right.\\
    &\quad \biggl. + C_{we} \omega(r(q))^{2} r(q)^m\biggr]r(q)^m\\
    \overset{(\ast)}&{\leq} 2^{-N+1}Mc^{-1}r^m\int_{\{\frac{c\bE}{\gamma}\leq\bd_T\leq1\}\cap\baseball{\bp(p)}{s+rc^{-\frac{1}{m}}}}\bd_T\diff\cH^m \\
    &\quad + C_{we}c^{-1}r^{m}\sum_{i=1}^M\sum_{q\in\mathcal{A}_i}\omega(r(q))^2\\
    &\leq 2^{-N+1}Mc^{-1}r^m\int_{\{\frac{c\bE}{\gamma}\leq\bd_T\leq1\}\cap\baseball{\bp(p)}{s+rc^{-\frac{1}{m}}}}\bd_T\diff\cH^m \\
    &\quad + C_{we}c^{-1}\omega(r)^2 r^{m},
\end{aligned}
\end{equation*}
where in \((\ast)\) we use that \(r(q)\leq rc^{-\frac{1}{m}}, c^{-1}\leq 1\) and \(M<2^{N-1}\). Now, we choose
\begin{equation*}
    \theta = \theta(m,n,\varepsilon_{we}) :=  -\log_{\gamma}\left(\frac{M}{2^{N-1}}\right),
\end{equation*} 
with this number settled and recalling that \(c^{-1}\leq 1\), we finish the proof of the claim.
\end{proof}

\subsection{Strong excess decay}

We now enunciate the strong excess decay which is a straightened version of the weak statement, Lemma \ref{L:weak-excess-decay}, in which we cut off assumption \eqref{weak-decay:(c)} and improve \eqref{E:weak-excess-decay} to \eqref{T:strong-decay}. This stronger decay will allow us to improve all our previous approximation to then give the proof of the main theorem, Theorem \ref{T:strong-approx}. 

\begin{thm}[Strong excess decay]\label{T:strong-excess-decay}
There exist positive geometric constants \(\varepsilon_{se}=\varepsilon_{se}(m,n,Q)>0,\) \(C_{se}=C_{se}(m,n,Q)>0,\) and \(\gamma_{se}=\gamma_{se}(m,n,Q)>0\) such that, if
\begin{enumerate}[\upshape (a)]
    \item \(T\) is \(\omega\)-almost area minimizing under Assumption \ref{A:LA} and {$\int_0^{s_\omega}\frac{\omega(\rho)}{\rho}\mathrm{d}\rho < +\infty$,}
    
    \item \(\bE(T,\cyl{p}{4r})<\varepsilon_{se}<\varepsilon_a\),
\end{enumerate}
hold, thus, for every Borel set \(A\subset \baseball{\bp(p)}{\frac{5r}{4}}\), we have that
\begin{equation}\label{T:strong-decay}
\be_T(A) \leq C_{se}\left(\bE(T,\cyl{p}{4r})^{\gamma_{se}} + \cH^m(A)^{\gamma_{se}}\right)\left(\bE(T,\cyl{p}{4r}) + \omega(r)^{2}\right)r^m.
\end{equation}
\end{thm}
\begin{proof}
The proof of this result is similar to the one given in \cite[Theorem 7.1]{DS3} where the authors consider area minimizing integral currents with support contained in an ambient manifold. We set \(\bE:=\bE(T,\cyl{p}{4r})\). Given \(\beta\in(0,\frac{1}{2m})\), as it is done in \cite[Proposition 7.3]{DS3}, using a regularization by convolution technique, we can build up a set \(B\subset [r,2r]\) with \(\cH^1(B)>\sfrac{r}{2}\) such that, for every \(z\in B\), there exists a Lipschitz function \(g:\baseball{\bp(p)}{z}\to\Is{Q}\) satisfying
\begin{align}
    \operatorname{Lip}(g) &\leq C_1\bE^{\beta}, \nonumber \\
    g|_{\partial\baseball{\bp(p)}{z}} &= f|_{\partial\baseball{\bp(p)}{z}}, \nonumber \\
    \int_{\baseball{\bp(p)}{z}}|Dg|^2\diff\cH^m &\leq \int_{\baseball{\bp(p)}{z}\cap K}|Df|^2\diff\cH^m  + C_1\bE^{\gamma_0}\left(\bE+\omega(r)^2\right)r^m, \label{E:strong-decay:Dir-g-Dir-f}
\end{align}
where \(f\) and \(K\) are given by the weak approximation from Proposition \ref{P:weak-lip}, and \(\gamma_0 = \gamma_0(m,n,Q)>0\) and \(C_1 = C_1(m,n,Q)>0\) are positive geometric constants. There is a radius \(z\in B\) and a current \(P\in\mathbf{I}_m(\R^{m+n})\) with
\begin{align}
\partial P &= \partial \left(T\res\cyl{p}{z}-\mathbf{G}_f\res\cyl{p}{z}\right) \label{E:strong-decay:bdry-P-equal} \\
\| P\|(\cyl{p}{z}) &\leq C_2 \bE^{1+\gamma_1}r^m. \label{E:strong-decay:mass-P}
\end{align}

We now take \(\gamma\) to be \(\min\{\gamma_0, \gamma_1\}\). Thanks to \eqref{E:strong-decay:bdry-P-equal} and Lemma \ref{L:homotopy}, we are apt to apply \cite[Equation 2.2]{DSS2} in our context which, together with \eqref{E:strong-decay:mass-P}, provides us the following
\begin{equation}\label{E:from-2-homot}
\|T\|(\cyl{p}{z}) \leq \|\mathbf{G}_g\|(\cyl{p}{z})+C_3\bE^{1+\gamma}r^m + C_3\omega(r)^2r^{m}\bE^{\frac{3}{4}} + C_3\int_{\baseball{\bp(p)}{z}}\cG(g,f)\diff\cH^m.
\end{equation}

We proceed with a simple algebraic argument, notice that for any nonzero real numbers \(a\) and \(b\), we have that \(0\leq (ba)^2 -2ab + 1\) which leads to \(2a \leq ba^2 + \sfrac{1}{b}\). Therefore, taking \(a=\omega(r)^2r^{m}\bE^{\sfrac{3}{4}}\) and \(b = \omega(r)^{-2}r^{-m}\bE^{\gamma-\sfrac{3}{2}}\), we have that
\begin{equation*} 2\omega(r)^2r^{m}\bE^{\frac{3}{4}} \leq \omega(r)^2r^{m}\bE^{\gamma} + \omega(r)^2r^{m}\bE^{\frac{3}{2}-\gamma} \leq \omega(r)^2r^{m}\bE^{\gamma} + \omega(r)^2r^{m}\bE^{1+\gamma},\end{equation*}
where in the last inequality we used that \(\gamma < \sfrac{1}{4}\). By the latter inequality and \eqref{E:from-2-homot}, we get that 
\begin{equation}\label{E:strong-decay:massT}
\|T\|(\cyl{p}{z}) \leq \|\mathbf{G}_g\|(\cyl{p}{z}) + C_3\bE^\gamma\left(\omega(r)^2 + \bE\right)r^m + C_3\int_{\baseball{\bp(p)}{z}}\cG(g,f)\diff\cH^m.
\end{equation}

Now, we apply the Taylor expansion, \cite[Corollary 3.3]{DS2}, for \(\bG_g\) together with \eqref{E:strong-decay:Dir-g-Dir-f} to get that
\begin{equation}\label{E:strong-decay:massGg}
\|\bG_g\|(\cyl{p}{z}) \leq Q\cH^m(\baseball{\bp(p)}{z}) + \int_{\baseball{\bp(p)}{z}\cap K}\frac{|Df|^2}{2} + C_1\bE^{\gamma}\left(\bE+\omega(r)^2\right)r^m,
\end{equation}
and for \(\bG_f\) to obtain that
\begin{equation}\label{E:strong-decay:massGf}
\|\bG_f\|(\cyl{p}{z}\cap K) \geq Q\cH^m(\baseball{\bp(p)}{z}\cap K) + \int_{\baseball{\bp(p)}{z}\cap K}\frac{|Df|^2}{2} - C_4\bE^{\gamma}\left(\bE+\omega(r)^2 r^{m}\right).
\end{equation}

So, in order to estimate the excess measure of the bad set, we put the past inequalities into account as follows:
\begin{align}
    \be_T(\baseball{\bp(p)}{z}\setminus K) &= \| T\|(\cyl{p}{z}\setminus K) - Q\cH^m(\baseball{\bp(p)}{z}\setminus K) \nonumber  \\
    &= \|T \|(\cyl{p}{z}) - \|T \|(\cyl{p}{z}\cap K) - Q\cH^m(\baseball{\bp(p)}{z}\setminus K) \nonumber \\
    \overset{\eqref{E:strong-decay:massT},\eqref{E:strong-decay:massGg}}&{\leq} Q\cH^m(\baseball{\bp(p)}{z}) - Q\cH^m(\baseball{\bp(p)}{z}\setminus K)  \nonumber\\
    &\quad +\int_{\baseball{\bp(p)}{z}\cap K}\frac{|Df|^2}{2} + C_1\bE^{\gamma}\left(\bE+\omega(r)^{2}\right)r^m \nonumber \\
    &\quad + C_3\bA\int_{\baseball{\bp(p)}{z}}\cG(g,f)\diff\cH^m - \|T \|(\cyl{p}{z}\cap K) \nonumber \\
    \overset{\eqref{E:strong-decay:massGf}}&{\leq} C_5\bE^{\gamma}\left(\bE+\omega(r)^2\right)r^m + C_3\bA\int_{\baseball{\bp(p)}{z}}\cG(g,f)\diff\cH^m \nonumber \\
    &\quad + \|\bG_f\|(\cyl{p}{z}\cap K) - \|T \|(\cyl{p}{z}\cap K) \nonumber \\
    &= C_5\bE^{\gamma}\left(\bE+\omega(r)^2\right)r^m + C_3\bA\int_{\baseball{\bp(p)}{z}}\cG(g,f)\diff\cH^m \nonumber ,
\end{align}
where in the last inequality we use that in the good set \(K\) the current \(\bG_f\) induced by the weak approximation \(f\) of the current \(T\) coincides with the current \(T\), i.e. \eqref{E:weak-app:=inK}. Now we notice that
\begin{equation*} \int_{\baseball{\bp(p)}{z}}\cG(g,f)\diff\cH^m \leq C_6 \bE^{\gamma}\left(\bE + \omega(r)^2\right)r^m, \end{equation*}
the proof of this fact is given in the proof of \cite[Theorem 4.1]{DS3}, the argument given by the authors works line by line in our setting. The latter inequality together with the long chain of inequalities above furnish
\begin{equation}\label{E:strong-decay:e_TinBadSet}
\be_T(\baseball{\bp(p)}{z}\setminus K) \leq C_7 \bE^{\gamma}\left(\bE + \omega(r)^2\right)r^m.  
\end{equation}

Let us now handle the term \(\be_T(A)\). To that end, we recall that \(|Df|^2\) is a \(L^1\) function, so, almost every point of \(K\) is a Lebesgue point of \(|Df|^2\) which together with the Taylor expansion ensure that
\begin{equation}\label{E:strong-decay:Df}
\begin{aligned}
|Df|^2(x)&=\lim_{t\to 0}\frac{1}{\omega_mt^m}\int_{\baseball{x}{t}\cap K}|Df|^2\diff\cH^m \leq C_8  \lim_{t\to 0}\frac{\be_{\mathbf{G}_f}(\baseball{x}{t}\cap K)}{\omega_mt^m} \\
\overset{\eqref{E:weak-app:=inK}}&{=} C_8  \lim_{t\to 0}\frac{\be_{T}(\baseball{x}{t}\cap K)}{\omega_mt^m} \leq C_8 \limsup_{t\to 0}\frac{\be_{T}(\baseball{x}{t})}{\omega_mt^m} = C_8 \bd_T(x).
\end{aligned}
\end{equation}

We now are in position to apply Proposition \ref{P:higher-int} and we also recall that\(\bd_T \leq C_9 \bE^{1+\gamma} < 1\) in \(K\). Therefore, given a Borel set \(A\subset\baseball{\bp(p)}{z}\), we proceed as follows:
\begin{equation}\label{E:strong-decay:eT(A)}
\begin{aligned}
    \be_T(A\cap K) \overset{\textup{Taylor}}&{\leq} \int_{A\cap K}\frac{|Df|^2}{2}\diff\cH^m + C_{10}\bE^{1+\gamma} \\
    \overset{\textup{H\"older ineq.}}&{\leq} \cH^m(A\cap K)^{\frac{a-1}{a}}\left(\int_{A\cap K}\frac{|Df|^{2a}}{2}\diff\cH^m \right)^{\frac{1}{a}} + C_{10}\bE^{1+\gamma} \\
    \overset{\eqref{E:strong-decay:Df}}&{\leq} \frac{C_8}{2^{\frac{1}{a}}} \cH^m(A)^{\frac{a-1}{a}}\left(\int_{A\cap K}\bd_T^{a}\diff\cH^m \right)^{\frac{1}{a}} + C_{10}\bE^{1+\gamma} \\
    &\leq C_{8} \cH^m(A)^{\frac{a-1}{a}}\left(\int_{\{\bd_T\leq 1\}\cap\baseball{\bp(p)}{z}}\bd_T^{a}\diff\cH^m \right)^{\frac{1}{a}} + C_{10}\bE^{1+\gamma} \\
    \overset{\eqref{E:thm:higher-int}}&{\leq} C_8C_a\cH^m(A)^{\frac{a-1}{a}}  \left[\bE^a r^m + \omega(r)^2 r^{m}\bE^{a-1}\right]^{\frac{1}{a}} + C_{10}\bE^{1+\gamma} \\
    &\leq C_8C_a\cH^m(A)^{\frac{a-1}{a}} \left[\bE + \omega(r)^2\right]r^m + C_{10}\bE^{1+\gamma},
\end{aligned}
\end{equation}
where in the last inequality we used (c) to obtain that \(\bE\left[1 + \omega(r)^2\bE^{-1}\right]^{1/a}\) is close to \(\bE\) and consequently smaller than \(\bE + \omega(r)^2\).
Finally, recalling that \(z>r\) and \(A\subset\baseball{\bp(p)}{z}\), possibly choosing \(\gamma\) smaller depending on \(a\), we put \eqref{E:strong-decay:eT(A)}, the triangle inequality and \eqref{E:strong-decay:e_TinBadSet} together to easily conclude the proof of the theorem. Lastly, we remark that \(z\in B\) and \(\cH^1(B)>\frac{r}{2}\), then we surely can take \(z\in\left[\frac{5}{4}r,2r\right].\)
\end{proof}

\subsection{Superlinear Lipschitz approximation}

We now state our main strong approximation theorem which is the analogous of \cite[Theorem 1.4]{DSS2} to our general almost minimality condition, and we also provide some improvements.

\begin{thm}[Strong Lusin type Lipschitz approximation]\label{T:strong-approx}
There exist positive geometric constants \(\varepsilon_{la}=\varepsilon_{la}(m,n,Q)>0,\) \(C=C(m,n,Q)>0,\) and \(\gamma_{la}=\gamma_{la}(m,n,Q)\in (0,1/(2m))\) such that, if
\begin{enumerate}[\upshape (a)]
    \item
    \(T\) is \(\omega\)-almost area minimizing under Assumption \ref{A:LA} and {$\int_0^{s_\omega}\frac{\omega(\rho)}{\rho}\mathrm{d}\rho < +\infty$,}
    
    \item\label{strong-approx:(b)} \(\bE(T,\cyl{p}{4r})<\varepsilon_{la}\),
\end{enumerate}
hold, then there exist a Lipschitz map \(f:\baseball{\bp(p)}{r}\to\Is{Q}\) and \(K\subset\baseball{\bp(p)}{r}\) such that
\begin{align}
\operatorname{Lip}(f) &\leq C\bE(T,\cyl{p}{4r})^{\gamma} \label{T:strong-approx:lip},\\
\mathbf{G}_f\res\left(K\times\R^n\right) &= T\res\left(K\times\R^n\right) \label{T:strong-approx:T=Gf},\\
\cH^m\left(\baseball{\bp(p)}{r}\setminus K\right) &\leq  C \bE(T,\cyl{p}{4r})^{\gamma}\left[\bE(T,\cyl{p}{4r}) + \omega^2(r)\right]r^m \label{T:strong-approx:bad-set},\\
\be_T(A\setminus K) &\leq  C \bE(T,\cyl{p}{4r})^{\gamma}\left[\bE(T,\cyl{p}{4r}) + \omega^2(r)\right]r^m \label{T:strong-approx:excess-measure},\\
\int_{A\setminus K}\frac{|Df|^2}{2}\diff\cH^m &\leq  C \bE(T,\cyl{p}{4r})^{3\gamma}\left[\bE(T,\cyl{p}{4r}) + \omega^2(r)\right]r^m \label{T:strong-approx:Df-bad-set},\\
\left|\be_T(A) -  \int_{A}\frac{|Df|^2}{2}\diff\cH^m \right| &\leq  C \bE(T,\cyl{p}{4r})^{\gamma}\left[\bE(T,\cyl{p}{4r}) + \omega^2(r)\right]r^m \label{T:strong-approx:Df-eT},
\end{align}
where the last inequality holds for every Borel set \(A\subset\baseball{\bp(p)}{r}\).
\end{thm}

We will usually refer to the set \(K\) as the \textbf{good set} and \(\baseball{\bp(p)}{r}\setminus K\) as the \textbf{bad set}.

\begin{proof}[Proof of Theorem \ref{T:strong-approx}]
Fix \(\varepsilon_{la}=:\varepsilon\) and \(\gamma_{la}=:\gamma\). We begin noticing that \(\varepsilon\) is supposed to be smaller than \(\varepsilon_{se}\), in particular, it is smaller than \(\varepsilon_{ha}\) and thus condition \eqref{strong-approx:(b)} is satisfied with \(\varepsilon_{ha}\). We start taking \(\gamma< \sfrac{1}{2m}\) and thus we take \(K\) and \(f\) to be the weak approximation given by Proposition \ref{P:weak-lip}. So, it is clear that \eqref{T:strong-approx:lip} and \eqref{T:strong-approx:T=Gf} follows respectively from \eqref{E:weak-app:lip-bound} and \eqref{E:weak-app:=inK}. In order to prove \eqref{T:strong-approx:bad-set}, we define 
\begin{equation*}A:= \{\mathbf{m}\be_T>2^{-m}\bE(T,\cyl{p}{4r})^{2\gamma}\}\cap\baseball{\bp(p)}{\sfrac{9r}{8}},\end{equation*}
notice that \cite[Proposition 3.2]{DS3} gives that \(\cH^m(A)\leq C_1\bE(T,\cyl{p}{4r})^{1-2\gamma}\). Therefore, possibly taking \(\gamma <\gamma_{se}\), we use the strong excess decay to obtain
\begin{align*}
\cH^m(\baseball{\bp(p)}{r}\setminus K) &\leq C_2\bE(T,\cyl{p}{4r})^{-2\gamma}\be_T(A)\\
\overset{\eqref{T:strong-decay}}&{\leq} C_2\bE(T,\cyl{p}{4r})^{2\gamma_{se}-2\gamma(1+\gamma_{se})}\left(\bE(T,\cyl{p}{4r}) + \omega(r)^{2}\right)r^m,
\end{align*} 
which, possibly choosing \(\gamma < \frac{\gamma_{se}}{2(1+\gamma_{se})}\), furnishes \eqref{T:strong-approx:bad-set}. We fix \(\bE(T,\cyl{p}{4r})=\bE\). Now, we firstly prove \eqref{T:strong-approx:Df-eT} for \(A=\baseball{\bp(p)}{r}\) as follows:
\begin{align*}
    \biggl| \be_T(\baseball{\bp(p)}{r}) -  \int_{\baseball{\bp(p)}{r}}\frac{|Df|^2}{2}\diff\cH^m \biggr| &= \biggl|\be_T(\baseball{\bp(p)}{r}) + \be_{\mathbf{G}_f}(\baseball{\bp(p)}{r}) \biggr.\\
    &\quad -\biggl. \be_{\mathbf{G}_f}(\baseball{\bp(p)}{r}) -  \int_{\baseball{\bp(p)}{r}}\frac{|Df|^2}{2}\diff\cH^m \biggr|
\end{align*}

Using that \(\mathbf{G}_f\) is equal to \(T\) on the good set, we obtain that
\begin{align*}
    \left|\be_T(\baseball{\bp(p)}{r}) -  \int_{\baseball{\bp(p)}{r}}\frac{|Df|^2}{2}\diff\cH^m \right| &\leq \be_T(\baseball{\bp(p)}{r}\setminus K) + \be_{\mathbf{G}_f}(\baseball{\bp(p)}{r}\setminus K) \\
    &\quad +\left|\be_{\mathbf{G}_f}(\baseball{\bp(p)}{r}) -  \int_{\baseball{\bp(p)}{r}}\frac{|Df|^2}{2}\diff\cH^m \right| \\
    \overset{\eqref{T:strong-decay},\eqref{T:strong-approx:bad-set}}&{\leq} C_2 \bE^{\gamma}\left[\bE + \omega(r)^2\right]r^m \\
    &\quad +\left|\be_{\mathbf{G}_f}(\baseball{\bp(p)}{r}) -  \int_{\baseball{\bp(p)}{r}}\frac{|Df|^2}{2}\diff\cH^m \right|\\
    \overset{\textup{Taylor}}&{\leq} C_3 \bE^{\gamma}\left[\bE + \omega(r)^2\right]r^m.
\end{align*}

For every Borel set \(A\subset\baseball{\bp(p)}{r}\), notice that \eqref{T:strong-approx:lip} and \eqref{T:strong-approx:bad-set} give
\begin{equation*}\int_{F\setminus K} \abs{Df}^2 \overset{\eqref{T:strong-approx:lip}}{\le} C_4 \bE^{2\gamma} \cH^m(\baseball{\bp(p)}{r}\setminus K) \overset{\eqref{T:strong-approx:bad-set}}{\le} C_5 \bE^{3\gamma}\left[\bE + \omega(r)^2\right]r^m , \end{equation*}
hence we achieve \eqref{T:strong-approx:Df-bad-set}. The Taylor expansion of the area functional, \cite[Corollary 3.3]{DS2} gives
\begin{equation*}\left|\be_{\bG_f}(A)-\frac12\int_A \abs{Df}^2\right| \le C_5 \bE^{\gamma}\left[\bE + \omega(r)^2\right]r^m.
\end{equation*}

Therefore, we obtain for every Borel set \(A\subset\baseball{\bp(p)}{r}\) that
\begin{align*}
    \be_T(A\setminus K) &= \be_T(A) - \be_{\bG_f}(A\cap K) \\
    \overset{\textup{triangle ineq.}}&{\le} \left|\be_T(A)-{\frac12} \int_{A} \abs{Df}^2 \right| + \left|{\frac12} \int_{A\cap K} |Df|^2- \be_{\bG_f}(A\cap K)\right|+ \int_{A\setminus K} \abs{Df}^2\\
    \overset{\eqref{T:strong-approx:Df-eT},\eqref{T:strong-approx:Df-bad-set}}&{\le}  C_6 \bE^{\gamma}\left[\bE + \omega(r)^2\right]r^m + \left|{\frac12} \int_{A\cap K} |Df|^2- \be_{\bG_f}(A\cap K)\right|\\
    \overset{\textup{Taylor}}&{\leq} C_7 \bE^{\gamma}\left[\bE + \omega(r)^2\right]r^m,
\end{align*}
which is \eqref{T:strong-approx:excess-measure}. With the aim at proving \eqref{T:strong-approx:Df-eT}, which we have proved above for the special case \(A=\baseball{\bp(p)}{r}\), we proceed as follows:
\begin{equation*}
\begin{aligned}
    \left| \be_T(A) - \frac12 \int_A \abs{Df}^2\right| &\le \left|\be_T(A\cap K) - \frac12 \int_{A\cap K} \abs{Df}^2\right|+ \be_T(A\setminus K) + \frac12 \int_{A\setminus K} \abs{Df}^2\\
    \overset{\eqref{T:strong-approx:Df-eT},\eqref{T:strong-approx:excess-measure},\eqref{T:strong-approx:Df-bad-set}}&{\le} C_8 \bE^{\gamma}\left[\bE + \omega(r)^2\right]r^m .
\end{aligned}
\end{equation*}
\end{proof}

\renewcommand{\abstractname}{Acknowledgements}
\begin{abstract}
The author was supported by CAPES-Brazil with a Ph.D. scholarship 88882.377954/2019-01. This work was carried out in Princeton University (this visit was financed by FAPESP-Brazil grant 2021/05256-0) and The Fields Institute for Research in Mathematical Sciences during the Thematic Program on Nonsmooth Riemannian and Lorentzian Geometry. My sincere thanks to Stefano Nardulli for reading a first version of this note and giving relevant insights. {I also thanks the anonymous referee for several important comments and suggestions}. 
\end{abstract}

            \bibliographystyle{siam}
            \addcontentsline{toc}{section}{References} 
            \bibliography{biblio}

\noindent\textit{Reinaldo Resende,\\
rresende@andrew.cmu.edu,\\
Carnegie Mellon University,\\
Wean Hall 7124,\\
Pittsburgh, PA, USA.}
\end{document}